\documentclass{article}

\usepackage{amsmath,amssymb,amsthm,enumerate,mathrsfs}
\usepackage[utf8]{inputenc}
\usepackage{empheq}
\usepackage{algpseudocode}
\usepackage{algorithmicx,algorithm}
\usepackage{mathtools}
\usepackage{color}
\usepackage{graphicx}
\usepackage{float}
\usepackage{caption}
\usepackage{tabularx,array}
\usepackage[nice]{nicefrac}
\usepackage[colorlinks=true, allcolors=blue]{hyperref}
\usepackage{subcaption}

\newtheorem{theorem}{Theorem}[section]
\newtheorem{lemma}[theorem]{Lemma}

\newtheorem{proposition}[theorem]{Proposition}
\newtheorem{corollary}[theorem]{Corollary}
\newtheorem{definition}[theorem]{Definition}
\DeclareMathOperator{\argmin}{argmin}
\newcommand{\R}{\mathbb{R}}

\newcommand{\algoclass}{{\tt DFi}}
\newcommand{\DS}{{\tt ds}}
\newcommand{\MB}{{\tt mb}}
\newcommand{\Eds}{\mathbb{E}_{\DS}}
\newcommand{\Emb}{\mathbb{E}_{\MB}}
\newcommand{\Sd}{\mathcal{S}^{d-1}}
\newcommand{\Od}{O(d)}
\newcommand{\V}{\mathcal{V}}
\newcommand{\T}{\mathrm{T}}
\newcommand{\cI}{\mathcal{I}}
\renewcommand{\top}{\mathrm{T}}
\newcommand{\E}[1]{\operatorname{\mathbb{E}}\left[ #1 \right]}
\newcommand{\EE}[2][]{\operatorname{\mathbb{E}}_{#1}\left[ #2\right]} 

\newcommand{\boxedeq}[2]{\begin{empheq}[box={\fboxsep=6pt\fbox}]{align}\end{empheq}}
\algnewcommand{\Inputs}[1]{%
  \State \textbf{Inputs:}
  \Statex \hspace*{\algorithmicindent}\parbox[t]{.8\linewidth}{\raggedright #1}
}
\algnewcommand{\Initialize}[1]{%
  \State \textbf{Initialize:}
  \Statex \hspace*{\algorithmicindent}\parbox[t]{.8\linewidth}{\raggedright #1}
}
\algnewcommand{\Stop}[1]{%
  \State \textbf{Stopping Conditions:}
  \Statex \hspace*{\algorithmicindent}\parbox[t]{.8\linewidth}{\raggedright #1}
}

\numberwithin{equation}{section}

\title{Expected decrease for derivative-free algorithms using 
random subspaces}
\author{
Warren Hare 
\thanks{Department of Mathematics, University of British 
Columbia, Kelowna, British Columbia, Canada.  Hare's research is partially 
funded by the Natural Sciences and Engineering Research Council of Canada 
(NSERC) grant RGPIN-2023-03555. 
ORCID 0000-0002-4240-3903 (\texttt{warren.hare@ubc.ca}).}
\and  
Lindon Roberts
\thanks{School of Mathematics and Statistics, University of Sydney, Camperdown NSW 2006, Australia. ORCID 0000-0001-6438-9703 (\texttt{lindon.roberts@sydney.edu.au}).} 
\and 
Cl\'ement W. Royer
\thanks{LAMSADE, CNRS, Universit\'e Paris Dauphine-PSL, 
Place du Mar\'echal de Lattre de Tassigny, 75016 Paris, France. Royer's 
research is partially funded by Agence Nationale de la Recherche through 
program ANR-19-P3IA-0001 (PRAIRIE 3IA Institute).
ORCID 0000-0003-2452-2172
(\texttt{clement.royer@lamsade.dauphine.fr}).}
}

\begin{document}
\maketitle


\begin{abstract}
Derivative-free algorithms seek the minimum of a given function based only on 
function values queried at appropriate points. Although these methods are 
widely used in practice, their performance is known to worsen as the problem 
dimension increases. Recent advances in developing randomized derivative-free 
techniques have tackled this issue by working in low-dimensional subspaces 
that are drawn at random in an iterative fashion. The connection between the dimension 
of these random subspaces and the algorithmic guarantees has yet to be fully 
understood.

In this paper, we develop an analysis for derivative-free algorithms (both direct-search and model-based 
approaches) employing random subspaces. Our results leverage linear local approximations of smooth 
functions to obtain understanding of the expected decrease achieved per 
function evaluation. Although the quantities of interest involve 
multidimensional integrals with no closed-form expression, a relative 
comparison for different subspace dimensions suggest that low dimension is 
preferable. Numerical computation of the quantities of interest confirm the 
benefit of operating in low-dimensional subspaces.
\end{abstract}


\section{Introduction}
\label{sec:intro}

Derivative-free algorithms are designed to minimize a function using solely 
function value information. These methods are 
particularly valuable for optimizing functions arising in complex engineering 
and learning models and, as such, have been applied in a diversity of 
fields~\cite{CAudet_WHare_2017,
ARConn_KScheinberg_LNVicente_2009b,JLarson_MMenickelly_SMWild_2019}.  
However, classical derivative-free algorithms typically struggle to 
optimize functions with a large number of variables, as they must explore a 
large variable space without the guidance provided by derivatives. For these 
algorithms, the number of function evaluations that are used at each iteration 
can scale linearly with the problem dimension. As a result, the use of 
derivative-free algorithms has historically been restricted to problems having no more 
than a hundred variables.

To overcome this fundamental limitation, recent algorithmic proposals have 
relied on applying iterations in randomly chosen subspaces. For example, several
derivative-free algorithms based on \emph{direct-search methods}
have been proposed that use opposite Gaussian directions (effectively 
a one-dimensional subspace) in various settings as a way to 
compute steps using no more than two function 
evaluations~\cite{EBergou_EGorbunov_PRichtarik_2020,
JCDuchi_MIJordan_MJWainwright_AWibisono_2015,
YuNesterov_VSpokoiny_2017}. Another line of work considered directions 
uniformly distributed in the unit 
sphere~\cite{MADinizEhrhardt_JMMArtinez_MRaydan_2008,
SGratton_CWRoyer_LNVicente_ZZhang_2015}. In that setting, it was shown that 
an almost-surely convergent algorithm could be designed by using only two 
function evaluations per iteration, with the best choice (both in terms of 
gradient approximation and practical performance) being to use opposite 
directions~\cite{SGratton_CWRoyer_LNVicente_ZZhang_2015}. More recently, a 
generalized analysis showed that random subspaces of arbitrary dimension 
could be used to design globally convergent 
methods~\cite{LRoberts_CWRoyer_2023}. Similar ideas were proposed in the 
context of finite-difference estimates aiming at approximating directional 
derivatives~\cite{DKozak_CMolinari_LRosasco_LTenorio_SVilla_2023,
DKozak_SBecker_ADoostan_LTenorio_2021}.

{\em Model-based derivative-free algorithms}, that operate by maintaining a 
model of the objective function, have also been revisited using random 
subspaces. A model-based trust-region algorithm was recently proposed in 
the context of nonlinear least squares~\cite{CCartis_LRoberts_2023}, drawing on similar ideas for derivative-based algorithms~\cite{CCartis_JFowkes_ZShao_2022,ZShao_2022}. 
A randomized subspace 
trust-region method was subsequently developed for stochastic 
optimization~\cite{KJDzahini_SMWild_2022}. We also note the use of 
sketching matrices within derivative-free trust-region methods as another 
setup in which random subspaces can be employed~\cite{MMenickelly_2023}.

In the direct-search setting, empirical performance strongly suggested 
that using one-dimensional subspaces provided the best 
results~\cite{SGratton_CWRoyer_LNVicente_ZZhang_2015,LRoberts_CWRoyer_2023}. 
The conclusions were not as definitive in the model-based case, where 
quadratic models seemingly required sufficiently large subspaces to be 
built in~\cite{CCartis_LRoberts_2023} (but that implementation incorporated numerous extra heuristics), while model-based algorithms using 
linear interpolation proved efficient using very low 
dimensions~\cite{KJDzahini_SMWild_2022}. Although convergence analysis 
often applies for subspaces of any sufficiently large---but still $O(1)$---dimension, it does not 
provide a clear understanding of the connection between subspace dimension 
and practical performance, nor why extremely low-dimensional spaces (e.g.~1 or 2) are good choices in practice.

In this paper, we examine expected decrease for derivative-free 
algorithms based on random subspaces. 
To our knowledge, our approach of quantifying the expected per-iteration and per-oracle-call objective decrease is a novel framework for studying the complexity of randomized methods for nonlinear optimization.
Our approach allows us to provide information about \emph{average-case} algorithm performance, instead of the more common worst-case performance analysis typical in complexity analysis (e.g.~\cite{SGratton_CWRoyer_LNVicente_ZZhang_2015,CCartis_LRoberts_2023}).

By considering a general algorithmic 
framework, we are able to handle both direct-search and model-based 
strategies. By leveraging local linear approximations of the function to 
minimize, we express our problem in terms of linear functions, which 
facilitates the derivation of decrease guarantees in expectation. Our 
analysis shows that using low subspace dimension leads to the best possible 
objective decrease per function evaluation. Since evaluating the objective is often the 
computational bottleneck of derivative-free algorithms, such a result further 
motivates the use of randomized subspaces in these methods.

The remainder of this paper is structured as follows. The rest of this 
introductory section sets the notations and recalls some useful results about 
uniform distributions in subspaces. Section~\ref{sec:algos} provides a general 
algorithm template that covers both direct-search and model-based techniques. 
Section~\ref{sec:theoDS} is dedicated to analyzing direct-search methods 
based on random subspaces. The corresponding results for model-based methods 
are described in Section~\ref{sec:theoMB}. Section~\ref{sec:num} illustrates 
our theoretical findings with numerical experiments. Finally, we discuss 
extensions of our results in Section~\ref{sec:conc}.

\subsection{Notations and probability background}
\label{ssec:notproba}

Throughout the paper, $d$ and $p$ will always denote integers greater than or 
equal to $1$ with $p \leq d$.   
The Euclidean norm in $\R^d$ will be denoted by $\|\cdot\|$.  
The identity matrix in $\R^{d \times d}$ will be denoted by $I_d$. The unit 
sphere in $\R^d$ will be denoted by $\Sd$. The set of orthogonal $d \times d$ 
matrices will be denoted by $\Od$. For $p\le d$, the Stiefel manifold of 
$p \times d$ matrices with orthogonal columns in $\R^d$ will be denoted by 
$\V_{p,d}:=\{X\in\R^{d\times p} : X^{\top} X = I_p\}$. Note that 
$\V_{1,d}$ corresponds to the unit sphere $S^{d-1}$ while $\V_{d,d}=\Od$.

Our main results will really heavily on uniform distributions within the 
Stiefel manifold. Key results about this distribution are gathered in the next 
lemma, and we omit the proofs as they can be found in reference textbooks 
on normed vector spaces~\cite[Section 1]{VMilman_GSchectman_1986} and 
manifolds~\cite[Section 2.2]{YChikuse_2003}.

\begin{lemma} 
\label{lem:stiefel} 
	For any integers $1 \le p \le d$, the following hold.
	\begin{enumerate}[(i)]
    	\item The uniform distribution on $V_{p,d}$ is uniquely defined. 
    	\item If $X$ follows a uniform distribution on $V_{p,d}$, then so 
    	does $Q_1 X Q_2^{\top}$ for any (possibly random) $Q_1\in V_{d,d}$ and 
    	$Q_2\in V_{p,p}$ independent of $X$. 
    	\item If $Q$ follows a uniform distribution on $V_{d,d}$, then so does 
    	$Q^{\top}$. 
    	\item We may construct $X\in V_{p,d}$ uniformly distributed by 
    	$X=Q_1 X_0 Q_2^{\top}$ for fixed $X_0\in V_{p,d}$, and independent and 
    	uniformly drawn $Q_1\in V_{d,d}$ and $Q_2\in V_{p,p}$. 
	\end{enumerate}
\end{lemma}

\section{Framework for derivative-free algorithms in random subspaces}
\label{sec:algos}

In this section, we present a general framework for a derivative-free algorithm that performs 
steps in randomly drawn subspaces. Section~\ref{ssec:genalgo} discusses our 
main framework, while Section~\ref{ssec:variants} gives two variations on the 
general method: one for direct-search methods and one for model-based methods.
Section~\ref{ssec:expdec} then defines the quantities of 
interest for analyzing the algorithms.

\subsection{General framework}
\label{ssec:genalgo}

Consider the minimization of a continuously differentiable function 
$f:\R^d \rightarrow \R$, where the derivative of $f$ cannot be used for 
algorithmic purposes. A derivative-free algorithm is an iterative procedure 
that explores the variable space by querying $f$ at finitely many points at 
every iteration in order to select the next iterate. In this paper, we are 
interested in derivative-free algorithms that produces such iterates by 
evaluating $f$ in a subspace of dimension $p \le d$ at every iteration, where 
this subspace is drawn randomly.

Algorithm~\ref{alg:DFAwRS} provides the general framework for our analysis. 
At each iteration, a random subspace is selected and one iteration of a 
given derivative-free method (\algoclass{}) is performed on that subspace. 
As our analysis focuses on expected decrease per iteration, we 
intentionally leave the stopping criterion and the step size update procedure 
undefined.

\begin{algorithm}[ht!]
\caption{Derivative-free algorithm with random subspaces}\label{alg:DFAwRS}
\begin{algorithmic}[1]
\Procedure{DFAwRS}{$f,x^0,\delta^0, p, \max_{fc}, \epsilon_{\tt stop},$
\algoclass}
    \State \% $f$: the objective function, $f:\R^d \mapsto \R$
    \State \% $x^0$: the initial point, $x^0 \in \R^d$
    \State \% $\delta^0$: the initial step size parameter, $\delta^0 > 0$
    \State \% $p$: the subspace dimension, $p \in \{1, 2, \ldots, n\}$ 
    \State \% \algoclass{}: iteration of the chosen DFO algorithm, used on 
    subspaces
    \While{stopping conditions not met}
    \State Randomly select a subspace of dimension $p$ with orthonormal 
    basis 
    $$B = \{b_1, b_2, ..., b_p\} \subseteq \R^d$$ 
    \State Define $f^k|_p : \R^p \rightarrow \R$ as 
    $f^k|_p (z) = f(x^k + \sum_{i=1}^p z_i b_i)$
    \State Create $z^*$ from the output of one iteration of \algoclass{} 
    applied to $f^k|_p$ using initial point $z^0=0$ and step size $\delta^k$
    \State Set $x^{k+1} = x^k + \sum_{i=1}^p z^*_i b_i$
    \State Select $\delta^{k+1}$ and increment $k \leftarrow k+1$
    \EndWhile
\EndProcedure
\end{algorithmic}
\end{algorithm}

In order to draw a random subspace at every iteration, we randomly generate a 
random orthonormal basis $B$ for a $p$-dimensional subspace of $\R^d$. In the 
rest of the paper, we will assume that $B$ is generated from the uniform 
distribution on the Stiefel manifold $V_{p,d}$, which amounts to taking the 
first $p$ columns of a uniformly sampled matrix from 
$\Od$~\cite[Section 2]{AEdelman_TAArias_STSmith_1998}. More precisely, given 
$Q=[q_1\ \cdots\ q_d]\in \Od$ uniformly sampled from the Haar measure, we 
construct the basis $B$ by letting $b_i=q_i$ for all $i=1,\dots,p$.

In the next section, we illustrate two variants on this method corresponding 
to the two main classes of derivative-free algorithms.

\subsection{Direct-search and model-based variants}
\label{ssec:variants}

Our first instance of \algoclass{} corresponds to a (directional) direct-search 
iteration. In their basic form, direct-search schemes do not attempt to build 
an approximate gradient, but merely explore the space along suitably chosen 
directions. In a deterministic setting, these directions usually form a 
positive spanning set, so that one of them is close to the steepest descent 
direction~\cite{TGKolda_RMLewis_VTorczon_2003}. Recent proposals in a 
probabilistic setting have replaced this requirement by random directions, 
with a particular interest for using directions belonging to a random 
subspace~\cite{SGratton_CWRoyer_LNVicente_ZZhang_2015,
LRoberts_CWRoyer_2023}. We adopt a similar approach in 
Algorithm~\ref{alg:DS} that describes our direct-search iteration.

\begin{algorithm}[ht!]
\caption{Direct-search iteration ({\tt ds})}\label{alg:DS}
\begin{algorithmic}[1]
\Procedure{\DS}{$f|_p ,z,\delta$}
    \State \% $f|_p $: the objective function, $f|_p :\R^p \mapsto \R$
    \State \% $z$: the incumbent solution, $z \in \R^p$
    \State \% $\delta$: the step size parameter, $\delta > 0$
    \State Consider the canonical basis $\{e_1, e_2, \ldots, e_p\}$ for 
    $\R^p$
    \State Return $z^* = \argmin \{ f|_p (z + \delta u) : 
    u \in \{\pm e_i\}_{i=1}^p \cup \{0\}\}$
\EndProcedure
\end{algorithmic}
\end{algorithm}

Note that we restrict ourselves to using coordinate directions in Algorithm \ref{alg:DS}. 
This is only to simplify presentation.  Indeed, applying Lemma \ref{lem:stiefel}(ii), it is clear that 
using a random orthonormal basis of $\R^p$ will produce the same expected decrease.
Note also that line 6 of Algorithm~\ref{alg:DS} states that complete 
polling is performed, that is we sample in all directions and return 
the best point that can be obtained. We will discuss how this algorithmic 
choice can be relaxed in Section~\ref{sec:theoDS}.

Our second algorithmic variant corresponds to a model-based iteration, and consists 
in building a linear interpolation model of the function. To this end, 
we leverage the notion of a simplex 
gradient~\cite[Chapter 9]{CAudet_WHare_2017}, which we restate below 
in a format tailored to our setup.

\begin{definition}
\label{de:simplexg}
	Consider the function $f|_p$ used in Algorithm~\ref{alg:DFAwRS}.
	Let $V= \begin{bmatrix} v^1 & v^2 & \ldots & v^p\end{bmatrix}$ be 
	an invertible matrix in $\R^{p \times p}$. For any $z \in \R^p$, the 
	{\em simplex gradient} of $f|_p$ at $z$ based on 
	$D$ is defined by
    \begin{equation}
    \label{eq:simplexg}
        \nabla_S f|_p (z, V) 
        = (V^\top)^{-1} 
        \begin{bmatrix}
        f|_p (z+v^1) -f|_p (z) \\
        f|_p (z+v^2) -f|_p (z) \\
        \vdots\\ 
        f|_p (z+v^p) -f|_p (z)
    	\end{bmatrix}
    	.
    \end{equation}
\end{definition}

The simplex gradient is used to construct a linear interpolation model of 
$f$, that can be used to produce a step from 
$z$~\cite[Chpt 9]{CAudet_WHare_2017}. This observation is at the heart of 
model-based derivative-free algorithms, where other, more elaborate models 
can be employed. In Algorithm~\ref{alg:MB}, we describe a trust-region 
model-based iteration based on a simplex gradient. This iteration computes 
a step that minimizes the model $u \mapsto \nabla_S f|_p(z,I_p)^\T u$ over 
a ball of radius $\delta$ centered at $z$, with $I_p$ being the identity 
matrix in $\R^{p \times p}$. In that simple case, the minimizer can be 
found explicitly, yielding formula~\eqref{eq:MBstep}.

\begin{algorithm}[ht!]
\caption{Model-based iteration ({\tt mb})}\label{alg:MB}
\begin{algorithmic}[1]
\Procedure{\MB}{$f|_p,z^0,\delta^0$}
    \State \% $f|_p$: the objective function, $f|_p:\R^p \mapsto \R$
    \State \% $z$: the incumbent solution, $z \in \R^p$
    \State \% $\delta$: the trust region radius, $\delta > 0$
    \State Evaluate $f|_p (z)$ and $f|_p (z+\delta e_i)$ ($i=1, 2, \ldots, p)$
    to construct
    \begin{equation}
    \label{eq:MBstep}
    	u = -\frac{\nabla_S f|_p (z, \delta {\tt I_p})}
    	{\|\nabla_S f|_p (z, \delta {\tt I_p})\|}
    \end{equation}
    \State Return $z^* = \argmin \{ f|_p (z), ~ f|_p (z + \delta u)\}$
\EndProcedure
\end{algorithmic}
\end{algorithm}

Similarly to Algorithm~\ref{alg:DS}, Algorithm~\ref{alg:MB} employs the coordinate directions 
in order to simplify presentation.

\subsection{Expected decrease guarantees}
\label{ssec:expdec}

Derivative-free algorithms are commonly designed so as to drive the 
step size or trust-region parameter $\delta^k$ to zero as the algorithm 
unfolds. Consequently, providing guarantees associated to the linear 
\emph{Taylor model} of the function around any given point leads to 
guarantees about decrease in function values.  We present one such result in Proposition~\ref{pr:dec}.

\begin{proposition}
\label{pr:dec}
	Suppose that $f$ is continuously differentiable with $L$-Lipschitz 
	continuous gradient. Consider the $k$th iteration of 
	Algorithm~\ref{alg:DFAwRS}, and suppose that we find a random 
	unit direction $u \in \R^d$ such that
	\begin{equation}
	\label{eq:decgrad}
		\E{\nabla f(x^k)^\T u} \le -\gamma < 0,
	\end{equation}
	where the expectation is taken over the randomness in $u$.
	Then, for sufficiently small $\delta_k$,
	\begin{equation}
	\label{eq:dec}
		\E{f(x^k + \delta^k u)-f(x_k)} \le -\frac{\gamma}{2}\delta^k,
	\end{equation}
	where the expectation is taken over the randomness in $u$.
\end{proposition}

\begin{proof}
	By Taylor expansion and Lipschitz continuity, one has
	\[
		f(x^k+\delta^k u) 
		\le f(x_k) + \delta^k \nabla f(x^k)^\T u 
		+ \frac{L}{2}(\delta^k)^2 \|u\|^2 
		= f(x^k) + \delta^k \nabla f(x^k)^\T u
		+ \frac{L}{2}(\delta^k)^2.
	\]
	Taking expectations with respect to the randomness in $u$ 
	leads to
	\begin{eqnarray*}
		\E{f(x^k + \delta^k u)-f(x_k)} 
		&\le &\delta^k \E{ \nabla f(x^k)^\T u} 
		+ \frac{L}{2}(\delta^k)^2 \\
		&\le &-\gamma \delta^k + \frac{L}{2}(\delta^k)^2.
	\end{eqnarray*}
	As a result, \eqref{eq:dec} is satisfied as long as 
	$\delta^k < \frac{1}{L}$.
\end{proof}

Considering the consequences of Proposition \ref{pr:dec}, in the rest of the paper, we focus on the function
\begin{equation}
\label{eq:flin}
	f^{lin}(x) = g^\T x,
\end{equation}
where $g \in \R^n$. Analyzing such functions is significantly easier 
than the general nonlinear case. In particular, note that 
$f^{lin}(x+\delta d)-f^{lin}(x) = \delta g^\T x$ for any 
pair of vectors. As a result, the function variation scales linearly 
with $\|g\|$ and $\delta$. In addition, upon applying 
Algorithm~\ref{alg:MB}, note that any simplex gradient
(using a well-poised sample set) will always be equal to the actual gradient 
$g$, regardless of the value of 
$\delta$~\cite[Exer 9.4]{CAudet_WHare_2017}. Therefore, we also 
assume without loss of generality that $\delta=\|g\|=1$. 

We are interested in the expected decrease that one can achieve over 
one iteration of a derivative-free algorithm regardless of the value 
of $g$. This leads us to the following definition.

\begin{definition}
\label{def:Epn}
	Consider applying one iteration Algorithm~\ref{alg:DFAwRS} using either Algorithm~\ref{alg:DS} or 
	Algorithm~\ref{alg:MB}, denoted by \algoclass{} $\in \{{\tt ds}, {\tt mb} \}$, to a function 
	$f^{lin}|_p$ obtained from $f^{lin}$ defined in~\eqref{eq:flin} 
	with a vector $g$ uniformly distributed on the unit sphere $\Sd$,
	using $\delta=1$ and $p \le n$.
	We define the \emph{expected decrease} 
	$\mathbb{E}_{\algoclass{}}[p,d]$ as
	\begin{equation}
	\label{eq:Epn}
    	 \mathbb{E}_{\algoclass{}}[p,d] := \E{f^{lin}(x^{k})-f^{lin}(x^{k+1})},
	\end{equation}
	where the expected value is taken over $g$ and $B$.
\end{definition}

Our key results, presented in Sections~\ref{sec:theoDS} and~\ref{sec:theoMB}, 
aim at providing formulae for the quantity~\eqref{eq:Epn}. In both cases, we 
will see that $B$ does not influence the value of the expected 
decrease.

\section{Analysis in the direct-search setting}
\label{sec:theoDS}

In this section, we examine the expected decrease for an iteration 
described by Algorithm~\ref{alg:DS}. Our main result will be obtained in 
Section~\ref{ssec:edDS} using bounds on multidimensional integrals, and we 
will discuss consequences in Section~\ref{ssec:DSperfun} in terms of relative 
decrease per function evaluation.

\subsection{Expected decrease formula}
\label{ssec:edDS}

As a preliminary result, we show that the expected decrease produced by 
Algorithm~\ref{alg:DS} is independent of the random subspace basis $B$.

\begin{proposition} 
\label{prop:EequalsDS}
	Consider the linear function $f^{lin}$ with $g \sim \Sd$, and suppose 
	that Algorithm~\ref{alg:DFAwRS} is applied using Algorithm~\ref{alg:DS} 
	as \algoclass{} (which we denote by \algoclass{}=\DS). Then, for any 
	$k$, the expected decrease satisfies 
	\begin{equation}
	\label{eq:EequalsDS}
		\Eds[p,d] = \EE[\tilde{g} \sim \Sd]{\max_{i=1,\dots,p} |\tilde{g}_i|}.
	\end{equation}		
\end{proposition}

\begin{proof} 
We first note that $x^k = x^{k+1}$ only when $B$ is orthogonal to $g$, 
and this occurs with probability $0$. Therefore, without loss of generality 
we assume  that $x^k \neq x^{k+1}$ so that $f(x^k)-f(x^{k+1})>0$. In that 
case, letting $B=[b_1\ \cdots\ b_p]$, we have
\begin{eqnarray*}
	f(x^k)-f(x^{k+1}) 
	&= &\max_{i=1,\dots,p} g^{\top} x^k - g^{\top} (x^k \pm b_i) \\
	&= &\max_{i=1,\dots,p} |g^\top b_i| \\
	&= &\|B^\T g\|_{\infty}.
\end{eqnarray*}
As a result,
\begin{equation*}
	\Eds[p,d] 
	= 
	\EE[\substack{B \sim \V_{p,d} \\ g \sim \V_{1,d}}]{\|B^\T g\|_{\infty}}.
\end{equation*}
Let $I_{d,p} := [e_1, \ldots, e_p] \in \V_{p,d}$ be the matrix containing the 
first $p$ coordinate directions in $\R^d$. By Lemma~\ref{lem:stiefel}(iv), 
we have $B=Q I_{d,p}$ for some $Q \sim \V_{d,d}$. Moreover, by 
Lemma~\ref{lem:stiefel}(ii) and (iii), the random vector $Q^\T g$ follows the 
same distribution than $g$, i.e. uniform distribution in $\V_{1,d}$. 
Therefore, we obtain
\begin{eqnarray*}
	\Eds[p,d] 
	&= &\EE[\substack{B \sim \V_{p,d} \\ g \sim \V_{1,d}}]{\|B^\T g\|_{\infty}} \\
	&= &\EE[\substack{Q \sim \V_{d,d} \\ g \sim \V_{1,d}}]
	{\|I_{d,p}^\T Q^\T g\|_{\infty}} \\
	&= &\EE[\tilde{g} \sim \V_{1,d}]{\|I_{d,p}^\T \tilde{g}\|_{\infty}} \\
	&= &\EE[\tilde{g} \sim \V_{1,d}]{\max_{i=1,\dots,p} |\tilde{g}_i|},
\end{eqnarray*}
proving~\eqref{eq:EequalsDS}.
\end{proof}

We will now obtain a mathematical expression for the 
expectation~\eqref{eq:EequalsDS}. When $d=1$, we necessarily have $p=1$ and 
\begin{equation*}
    \Eds[1,1] = 1.
\end{equation*}
Our main result will thus focus on the case $d>1$. In general, the expected 
decrease formula is considerably more intricate, as it involves multiple Gamma 
functions as well as the solution to a complex trigonometry integral. 

\begin{theorem}
\label{thm:expdecDS}
	Under the assumptions of Proposition~\ref{prop:EequalsDS}, suppose 
	further that $d>1$. Then, the expected decrease is given by
	\begin{equation}
	\label{eq:expdecDS}
    	\Eds[p,d] = \frac{p}{2}\frac{2^p}{(\sqrt{\pi})^{p}} 
    	\frac{\Gamma(d/2)\Gamma(p/2+1/2)}{\Gamma(d/2+1/2)} \cI(p), 
    \end{equation}
    where $\cI(p)$ is given by $\cI(1) := 1$ and
    \begin{equation}
    \label{eq:IpexpdecDS}
    	\cI(p) := \int_{R(p)}\left[ \prod_{i=1}^{p-1} \sin^i(\varphi_i) \right] 
    	d\varphi_{p-1} \cdots d\varphi_{1}, \qquad \text{if $p>1$,} 
    \end{equation}
    with the integration region $R(p)$ is $\{\varphi_1 \in [\pi/4, \pi/2]\}$ if 
    $p=1$ and 
	\begin{equation}
	\label{eq:RpexpdecDS}
		\left\{ (\varphi_1,\dots,\varphi_{p-1}) \in 
		[\pi/4,\pi/2] \times \prod_{i=2}^{p-1} \left[
		\arctan\left(\prod_{j=1}^{i-1}\csc\varphi_j\right), 
		\frac{\pi}{2}\right] 
		\right\} 
	\end{equation}
	otherwise.
\end{theorem}

\begin{proof} 
By Proposition~\ref{prop:EequalsDS}, we seek to evaluate the 
expectation~\eqref{eq:EequalsDS}, i.e.,
\[
	\Eds[p,d] = \EE[\tilde{g} \sim \Sd]{\max_{i=1,\dots,p} |\tilde{g}_i|}.
\]
To this end, it suffices to evaluate the integral over the region 
\[
	R(p,d) := \left\{\, \tilde{g} \in \Sd
	\ \middle|\ 
	\tilde{g}_1 \ge \tilde{g}_i \ge 0 
	\quad \forall i=1,\dots,p \,\right\},
\]
i.e., vectors in the nonnegative orthant for which the first coordinate is 
the largest. By symmetry, one can construct $p 2^d$ similar regions with 
the same integral value by selecting a maximal absolute value coordinate 
and an orthant. Thus, integrating over $R(p,d)$ gives $1/(p2^d)$ of the 
total integral. Moreover, for any $ \tilde{g} \in R(p,d)$, we get the 
simplification
\[
    \max\{|\tilde{g}_1|,\ldots,|\tilde{g}_p|\} = \tilde{g}_1, 
\]
and therefore~\eqref{eq:EequalsDS} can be rewritten as
\begin{equation}
\label{eq:edDSRpd}
    \Eds[p,d] = \frac{p 2^d}{|\Sd|} \int_{R(p,d)} \tilde{g}_1 \: 
    dS(\tilde{g}), 
\end{equation}
where $dS$ is the surface element for $\Sd$ and $|\Sd|$ is the 
volume of the unit sphere in $\R^d$.

To evaluate \eqref{eq:edDSRpd}, we use hyperspherical coordinates 
$(\varphi_1,\ldots,\varphi_{d-1})$ for $\Sd$:
\begin{eqnarray*}
	x_d &= &\cos(\varphi_1), \\
	x_{d-1} &= &\sin(\varphi_1)\cos(\varphi_2), \\	
	&\vdots & \\
	x_2 &= &\sin(\varphi_1) \cdots \sin(\varphi_{d-2})\cos(\varphi_{d-1}), \\
	x_1 &= &\sin(\varphi_1) \cdots \sin(\varphi_{d-2})\sin(\varphi_{d-1}),
\end{eqnarray*}
with surface element
\begin{equation*}
	dS = \sin^{d-2}(\varphi_1) \sin^{d-3}(\varphi_2) 
	\cdots 
	\sin(\varphi_{d-2}) d\varphi_1 d\varphi_2 \cdots d\varphi_{d-1}.
\end{equation*}
Note that this choice is a reverse of the traditional ordering of the axes, 
that will result in a simpler proof.

With these coordinates, the constraints defining the region 
$R(p,d)\subset \Sd$ translate into the following constraints on 
$(\varphi_1,\ldots,\varphi_{d-1})$:
\begin{itemize}
    \item $\tilde{g}_i \geq 0$ yields $\varphi_i \in [0,\pi/2]$ 
    for all $i=1,\ldots,d-1$;
    \item $\tilde{g}_1 \geq \tilde{g}_2$ yields 
    $\sin(\varphi_{d-1}) \geq \cos(\varphi_{d-1})$, which simplifies to 
    $\varphi_{d-1} \geq \pi/4;$
    \item $\tilde{g}_1 \geq \tilde{g}_3$ yields 
    $\sin(\varphi_{d-2})\sin(\varphi_{d-1}) \geq \cos(\varphi_{d-2})$, 
    which simplifies to 
    \[
        \varphi_{d-2} \geq \arctan(\csc(\varphi_{d-1})).
    \]
\end{itemize}
By continuing the process, we obtain the following description of 
$R(p,d)$ when $p=1$:
\[
   R(p,d) = \left\{ \varphi_i \in [0,\pi/2] \qquad \forall i=1,\ldots,d-1
   \right\}.
\]
When $p\geq 2$, then $R(p,d)$ is defined via the constraints
\begin{subequations}
\begin{align}
    \varphi_{d-1} 
    &\in [\pi/4, \pi/2], \label{eq_R1} \\
    \varphi_{d-i} 
    &\in \left[\arctan\left(\prod_{j=1}^{i-1} \csc(\varphi_{d-j})\right), 
    \frac{\pi}{2}\right], 
    & i &= 2,\ldots,p-1, \label{eq_R2} \\
    \varphi_i &\in [0,\pi/2], 
    & i &= 1,\ldots,d-p. \label{eq_R3}
\end{align}
\end{subequations}
Thus, returning to equation~\eqref{eq:edDSRpd}, we find that
\begin{eqnarray*}
	\Eds[p,d] 
	&= &\frac{p 2^d}{|\Sd|} \int_{R(p,d)} \tilde{g}_1 \: dS(\tilde{g}), \\
    &= &\frac{p 2^d}{|\Sd|} \int_{R(p,d)} \left(
    \prod_{i=1}^{d-1} \sin(\varphi_i)\right) 
    \left(\prod_{i=1}^{d-2} \sin^{d-i-1}(\varphi_i)\right) 
    d\varphi_1 \cdots d\varphi_{n-1}, \\
    &= &\frac{p 2^d}{|\Sd|} \int_{R(p,d)} 
    \left(\prod_{i=1}^{d-1} \sin^{d-i}(\varphi_i)\right) 
    d\varphi_1 \cdots d\varphi_{d-1}.
\end{eqnarray*}

When $p=1$, the integral is fully separable, and we obtain
\begin{equation}
\label{eq:edDS1}
	\Eds[1,d] = \frac{2^d}{|\Sd|} \prod_{i=1}^{d-1} \left(
	\int_{0}^{\pi/2} \sin^{d-i}(\theta) d\theta\right). 
\end{equation}
When $p>1$, we can factor out the integration with respect 
to $\varphi_1,\ldots,\varphi_{d-p}\in[0,\pi/2]$, yielding
\begin{equation}
\begin{array}{rl}
    \Eds[p,d] 
    &= \displaystyle \frac{p 2^d}{|\Sd|} 
    \left[\prod_{i=1}^{d-p} \int_{0}^{\pi/2} \sin^{d-i}(\theta) d\theta\right] 
    \cdot \\
    & \displaystyle ~\qquad \int_{\hat{R}(p,d)} \left(
    \prod_{i=d-p+1}^{d-1} \sin^{d-i}(\varphi_i)\right) 
    d\varphi_{d-p+1} \cdots d\varphi_{d-1}, 
    \label{eq:edDSp}
\end{array}
\end{equation}
where the reduced integration region $\hat{R}(p,d)$ is parameterized 
by inclusions~\eqref{eq_R1} and~\eqref{eq_R2} only. 
For simplicity, we now relabel the variables $\varphi_{d-i} \mapsto \varphi_i$ 
in the reduced integral over $\hat{R}(p,d)$ so as to get
\begin{equation}
\begin{array}{rl}
    &\displaystyle \int_{\hat{R}(p,d)} \left(
    \prod_{i=d-p+1}^{d-1} \sin^{d-i}(\varphi_i)\right) 
    d\varphi_{d-p+1} \cdots d\varphi_{d-1}  \\
    &= \displaystyle  \int_{R(p)} \left(
    \prod_{i=1}^{p-1} \sin^{i}(\varphi_i)\right) 
    d\varphi_{p-1} \cdots d\varphi_1, 
    \label{eq:edDSint}
\end{array}
\end{equation}
where the integration region $R(p)$ is defined by $\varphi_1 \in [\pi/4,\pi/2]$ 
and~\eqref{eq:RpexpdecDS}.
Combining \eqref{eq:edDS1} for $p=1$ with \eqref{eq:edDSp} and 
\eqref{eq:edDSint} for $p>1$, we obtain overall that
\begin{equation}
\label{eq:edDSaux}
    \Eds[p,d] = \frac{p 2^d}{|\Sd|} \left[
    \prod_{i=1}^{n-p} \int_{0}^{\pi/2} \sin^{d-i}(\theta) d\theta\right] 
    \cI(p), 
\end{equation}
where $\cI(p)$ is defined in~\eqref{eq:IpexpdecDS}.

Finally, we can simplify~\eqref{eq:edDSaux} using the identity
\begin{equation}
\label{eq:sinusint}
	\int_{0}^{\pi/2} \sin^{d-i} \theta d\theta 
	= \frac{\sqrt{\pi}}{2} \frac{\Gamma((d-i)/2+1/2)}{\Gamma((d-i)/2+1)}, 
\end{equation}
for any $i=1,\dots,d-p$. We then obtain
\begin{eqnarray}
    \prod_{i=1}^{d-p} \int_{0}^{\pi/2} \sin^{d-i}(\theta) d\theta 
    &= &\frac{\pi^{(d-p)/2}}{2^{(d-p)}} \frac{\Gamma(d/2)}{\Gamma(d/2+1/2)} 
    \frac{\Gamma(d/2-1/2)}{\Gamma(d/2)} 
    \cdots \frac{\Gamma(p/2+1/2)}{\Gamma(p/2+1)}, \nonumber \\
    &= &\frac{\pi^{(d-p)/2}}{2^{(d-p)}}
    \frac{\Gamma(p/2+1/2)}{\Gamma(d/2+1/2)}. 
    \label{eq:prodsinint}
\end{eqnarray}
Finally, applying \eqref{eq:prodsinint} and 
$|\Sd| = \tfrac{2\pi^{d/2}}{\Gamma(d/2)}$ to~\eqref{eq:edDSaux}, we arrive at
\[
    \Eds[p,d] = \frac{p}{2} \frac{ 2^{p}}{(\sqrt{\pi})^{p}} 
    \frac{\Gamma(d/2)\Gamma(p/2+1/2)}{\Gamma(d/2+1/2)} \cI(p),
\]
which is the desired result.
\end{proof}

We observe that the expression~\eqref{eq:expdecDS} is separable in $p$ and 
$d$. This property allows for simplified expressions for certain values of 
$p$, and also results in simplifications while comparing two pairs of 
values for $(p,d)$ as only one of the two dimension varies. We summarize 
these observations in the corollary below.

\begin{corollary}\label{cor:p12DS}
	Let $d_1,d_2,p_1,p_2$ be integers greater than or equal to $1$ such that 
	$\max\{p_1,p_2\} \le \max\{d_1,d_2\}$. Then, the 
	following properties hold:
	\begin{enumerate}[(i)]
		\item $\Eds[1,d_1] 
		= \frac{1}{\sqrt{\pi}}\frac{\Gamma(d_1/2)}{\Gamma(d_1/2+1/2)}$;
		\item if $d_1>2$, then $\Eds[2,d_1] 
		= \frac{\sqrt{2}}{\sqrt{\pi}} \frac{\Gamma(d_1/2)}{\Gamma(d_1/2+1/2)}$;
    	\item $\frac{\Eds[p_1,d_1]}{\Eds[p_2,d_1]} 
    	= \frac{\Eds[p_1,d_2]}{\Eds[p_2, d_2]}$;
    	\item $\frac{\Eds[p_1,d_1]}{\Eds[p_1,d_2]} 
    	= \frac{\Eds[p_2,d_1]}{\Eds[p_2,d_2]}$.
    \end{enumerate}
\end{corollary}

\begin{proof}
	The proofs of (i) and (ii) follow directly from~\eqref{eq:expdecDS} by using 
	$\cI(1)=1$, $\Gamma(1) = 1$, 
	$\cI(2) = \int_{\pi/4}^{\pi/2} \sin(\varphi_1) d\varphi_1 = 1/\sqrt{2}$, and 
	$\Gamma(3/2) = \frac{\sqrt{\pi}}{2}$.
	
	The proofs of (iii) and (iv) exploit the separability of the expression 
	\eqref{eq:expdecDS}. For any pair $(p,d)$ of integers greater than or 
	equal to $1$, define
	\[
		E(p) = \frac{p}{2}\frac{ 2^{p}}{(\sqrt{\pi})^{p}} \Gamma(p/2 + 1/2)\cI(p) 
		\quad \mbox{and} \quad
		\hat{E}(d) = \frac{\Gamma(d/2)}{\Gamma(d/2+1/2)} 
	\]
	so that $\Eds[p,d]=E(p)\hat{E}(d)$. Then,
    \begin{equation}
    \label{eq:E1overE1}
    	\frac{\Eds[p_1, d_1]}{\Eds[p_2,d_1]} 
   		= \frac{E(p_1)}{E(p_2)} 
   		= \frac{\Eds[p_1, d_2]}{\Eds[p_2, d_2]},
    \end{equation}
    proving (iii), and
    \begin{equation*}
    	\frac{\Eds[p_1, d_1]}{\Eds[p_1, d_2]} 
    	= \frac{\hat{E}(d_1)}{\hat{E}(d_2)} 
    	= \frac{\Eds[p_2,d_1]}{\Eds[p_2,d_2]},
    \end{equation*}
    proving (iv).
\end{proof}

Another consequence of the separable nature of the 
expression~\eqref{eq:expdecDS} is that the asymptotic behaviour of this 
quantity as $d \rightarrow \infty$ depends entirely on $p$. To establish this 
property, we rely on the following lemma.

\begin{lemma}
\label{lem:asympt}
	Asymptotically, 
    \begin{equation*}
        \frac{\Gamma(d/2)}{\Gamma(d/2+1/2)} 
        \rightarrow 
        \frac{\sqrt{2}}{\sqrt{d}} 
        ~\mbox{as}~d\rightarrow \infty.
    \end{equation*}
\end{lemma}

\begin{proof} 
Gautschi's inequality \cite[Eq.~(5.6.4)]{NIST_DLMF} states
\begin{equation*}
	{\displaystyle x^{1-s}
	<{\frac {\Gamma (x+1)}{\Gamma (x+s)}}
	<\left(x+1\right)^{1-s},}
\end{equation*}
for all $x>0$ and $s\in(0,1)$. Setting $x=\frac{d}{2}$ and $s=\frac{1}{2}$ 
provides 
\begin{equation}
\label{eq:gautschiex}
    \frac{\sqrt{d}}{\sqrt{2}} 
    < \frac{\Gamma (d/2+1)}{\Gamma (d/2+1/2)}
    < \frac{\sqrt{d+2}}{\sqrt{2}}. 
\end{equation}
Applying $\Gamma(d/2+1) = \frac{d}{2}\Gamma(d/2)$ now shows 
\begin{equation*}
    \displaystyle 
    \frac{\sqrt{2}}{\sqrt{d}} 
    < {\frac{\Gamma (d/2)}{\Gamma (d/2+1/2)}} 
    < \frac{\sqrt{2}\sqrt{d+2}}{d}.
\end{equation*}
Passing to a limit provides the asymptotic. 
\end{proof}

Combining the result of Lemma~\ref{lem:asympt} with 
Corollary~\ref{cor:p12DS}(i) and (ii) leads to the following 
asymptotics.

\begin{corollary}
\label{cor:p12limitsDS}
	Under the same assumptions as Corollary~\ref{cor:p12DS}, 
	asymptotically 
	\begin{equation*}
		\Eds[1,d_1] \rightarrow \frac{\sqrt{2}}{\sqrt{\pi}\sqrt{d_1}} 
		~\mbox{as}~d_1\rightarrow \infty
	\end{equation*}
	and for $d_1 >2$, asymptotically
	\begin{equation*}
		\Eds[2,d_1] 
		\rightarrow \frac{2}{\sqrt{\pi}\sqrt{d_1}} 
		~\mbox{as}~d_1\rightarrow \infty.
	\end{equation*}
\end{corollary}

\subsection{Expected decrease per function evaluation}
\label{ssec:DSperfun}

Derivation-free algorithms are typically used in situations where function evaluations are considered to be expensive 
calculations. As such, the effectiveness of a derivative-free algorithm is not 
gauged by expected decrease per iteration, but expected decrease per function 
evaluation. We thus wish to account for this cost in our formula for expected 
decrease.

Returning to Algorithm \ref{alg:DS}, we assume that the function value of the 
incumbent solution $x^k$ is already known from the output of the previous 
iteration.  As such, one iteration of Algorithm \ref{alg:DS} will evaluate the 
function at $2p$ new points, where $p$ is the subspace dimension. In this 
section, we are thus interested in the quantity
\begin{equation}
\label{eq:eDSF}
	\Eds^F[p,d] := \frac{\Eds[p,d]}{2p}.
\end{equation}
Our goal is then to study the variation of the quantity $\Eds^F[p,d]$ as a 
function of $p$. In order to derive such a result, we require the following 
lemma.

\begin{lemma} 
\label{lem:intdecrDS}
	Let the assumptions of Theorem~\ref{thm:expdecDS} hold, and $\cI(p)$ be 
	defined as in this theorem. Then, for any $p \le d-1$,
	\begin{equation*}
	\label{eq:intdecrDS}
		\frac{2}{\sqrt{\pi}} \frac{\Gamma(p/2+1)}{\Gamma(p/2+1/2)} 
		< \frac{\cI(p)}{\cI(p+1)}.
	\end{equation*}
\end{lemma}

\begin{proof}
If $p=1$, using the values of $\cI(1)$ and $\cI(2)$ from the proof of 
Corollary~\ref{cor:p12DS} gives
\begin{equation*}
	\cI(p+1) = \cI(2) 
	= \frac{1}{\sqrt{2}} 
	< 1 
	= \frac{\sqrt{\pi}}{2}\frac{\Gamma(1)}{\Gamma(3/2)} \cI(1) 
	= \frac{\sqrt{\pi}}{2}\frac{ \Gamma(p/2+1/2)}{\Gamma(p/2+1)} \cI(p).
\end{equation*}

Suppose now that $p>1$. Using the definition of $\cI(p)$~\eqref{eq:IpexpdecDS}, 
we have
\begin{equation}
\label{eq:intdecrDSaux}
	\cI(p+1) 
	= \int_{R(p)} \int_{\arctan(\csc(\varphi_1)\cdots \csc(\varphi_{p}))}^{\pi/2} 
	\left[\prod_{i=1}^{p} \sin^i(\varphi_i)\right] 
	d\varphi_{p} d\varphi_{p-1} \cdots d\varphi_1,
\end{equation}
showing that $\cI(p+1)$ is formed by including an extra inner integral 
inside the expression for $\cI(p)$. We now bound the lower limit of region of 
integration for $\varphi_p$ through induction. From the definition of $R(p)$, 
we have $\varphi_1 \in [\pi/4,\pi/2]$. Inductively, suppose that the region of 
integration of $\varphi_i$ is a subset of $[\pi/4,\pi/2]$ for $i=1,\ldots,k-1$. 
In that case, we have $\csc(\varphi_i)\in[1,\sqrt{2}]$ for all 
$i=1,\ldots,k-1$ and so 
\[
	\varphi_{k} \geq \arctan(\csc(\varphi_1)\cdots\csc(\varphi_{k-1})) 
	\ge \arctan(1) = \pi/4,
\] 
which implies that the region of integration for $\varphi_k$ is a subset of 
$[\pi/4,\pi/2]$. 

By the principles of mathematical induction, we have thus established that the 
region of integration of $\varphi_{p}$ is a subset of $[\pi/4,\pi/2]$. 
Applying this to the region of integration approximation to 
\eqref{eq:intdecrDSaux}, we find
\begin{eqnarray*}
	\cI(p+1) 
	&\le &\int_{R(p)} \int_{\pi/4}^{\pi/2} \left[
	\prod_{i=1}^{p} \sin^i(\varphi_i)\right] 
	d\varphi_{p} d\varphi_{p-1} \cdots d\varphi_1, \\
	&= &\left(\int_{R(p)} \left[
	\prod_{i=1}^{p-1} \sin^i(\varphi_i)\right] 
	d\varphi_{p-1} \cdots d\varphi_1\right) \left(
	\int_{\pi/4}^{\pi/2} \sin^{p}(\varphi_{p}) d\varphi_{p}\right), \\
    &= &\cI(p) \int_{\pi/4}^{\pi/2} \sin^{p}(\theta) d\theta.
\end{eqnarray*}
The result now follows from
\[
	\int_{\pi/4}^{\pi/2} \sin^{p}(\theta) d\theta 
	< \int_{0}^{\pi/2} \sin^{p}(\theta) d\theta 
	= \frac{\sqrt{\pi}}{2} \frac{\Gamma(p/2+1/2)}{\Gamma(p/2+1)},
\]
where the last equality uses the identity \eqref{eq:intdecrDSaux}.
\end{proof}

Using the previous result, we can approximate the rate at which $\cI$ 
decreases as a function of $p$.

\begin{proposition}
\label{prop:intdecrDSsimple}
	Let the assumptions of Theorem~\ref{thm:expdecDS} hold, and $\cI(p)$ be 
	defined as in Theorem~\ref{thm:expdecDS}. Then, for any $p \le d-1$,
	\begin{equation*}
		\cI(p+1) < \frac{\sqrt{\pi}}{\sqrt{2}\sqrt{p}}  \cI(p).
	\end{equation*}
\end{proposition}

\begin{proof} 
By Gautschi's inequality (see equation \eqref{eq:gautschiex}), we have that 
\[              
	\frac{\Gamma(p/2+1/2)}{\Gamma(p/2+1)} < \frac{\sqrt{2}}{\sqrt{p}}.
\]
Combining this with Lemma~\ref{lem:intdecrDS} completes the proof.
\end{proof}

We can now prove that the expected decrease per function evaluation is a 
strictly decreasing function of $p$.

\begin{theorem}
\label{thm:expdecperfunDS} 
	Let the assumptions of Theorem~\ref{thm:expdecDS} hold, and $\cI(p)$ be 
	defined as Theorem~\ref{thm:expdecDS}. Then, for any $p \le d-1$,
	\begin{equation*}
		\frac{\Eds[p,d]}{2p} > \frac{\Eds[p+1,d]}{2(p+1)}
	\end{equation*}
\end{theorem}

\begin{proof}
It suffices to show that 
\[
	\frac{\Eds[p,d]}{\Eds[p+1,d]} > \frac{p}{p+1}.
\]
As in the proof of Corollary~\ref{cor:p12DS}, we define
\[
	E_1(p) = \frac{p}{2} \frac{2^{p}}{(\sqrt{\pi})^{p}} 
	\Gamma(p/2 + 1/2)\cI(p).
\]
From equation \eqref{eq:E1overE1}, we have 
\begin{eqnarray*}
	\frac{\Eds[p,d]}{\Eds[p+1,d]}
	&= &\frac{E_1(p)}{E_1(p+1)}\\
	&= &\frac{(p/2)( 2^{p}/\sqrt{\pi}^{p}) \Gamma(p/2+1/2) \cI(p)}
	{(p/2+1/2)( 2^{p+1}/\sqrt{\pi}^{p+1}) \Gamma(p/2+1) \cI(p+1)}\\
	&= &\frac{p}{p+1}\frac{\sqrt{\pi}}{2}\frac{\Gamma(p/2+1/2)}{\Gamma(p/2+1)}
	\frac{\cI(p)}{\cI(p+1)}\\
	&> &\frac{p}{p+1}\frac{\sqrt{\pi}}{2}\frac{\Gamma(p/2+1/2)}{\Gamma(p/2+1)}
	\frac{2}{\sqrt{\pi}} \frac{\Gamma(p/2+1)}{\Gamma(p/2+1/2)}\\
	&= &\frac{p}{p+1},
\end{eqnarray*}
where the strict inequality arises from applying Lemma~\ref{lem:intdecrDS}. 
\end{proof}

The result of Theorem~\ref{thm:expdecperfunDS} suggests that performing 
direct-search iterations is more beneficial with low-dimensional subspaces,  
and that $p=1$ provides the best return on investment. Although our result 
applies to a linear function, we emphasize again that it can be connected to 
general smooth functions through arguments such as that of 
Proposition~\ref{pr:dec}.

To end this section, we discuss how our analysis can be adapted to classical 
considerations for direct-search methods in practice.

\paragraph{Opportunistic polling:} 
In Algorithm~\ref{alg:DS}, all candidate points are sampled in order to select 
the best one, i.e., complete polling is performed. In a serial environment, a 
cheaper practice called opportunistic polling consists in accepting the first 
point that yields decrease. Remarkably, this strategy does not jeopardize 
convergence and can bring significant savings in 
practice~\cite{TGKolda_RMLewis_VTorczon_2003}\cite{LRoberts_CWRoyer_2023}. 

Our analysis can be adapted to account for opportunistic polling under the 
assumption that Algorithm~\ref{alg:DS} evaluates the directions in the order 
$\{e_1,-e_1,\ldots\}$ (or more generally by evaluating pairs of opposite 
directions consecutively). In that case, with probability $1$ either $e_1$ 
or $-e_1$ will lead to a decrease in $f^{lin}$, and thus the step will be 
accepted. The expected decrease guarantees are therefore equivalent to those 
in the case $p=1$. In fact, one can go one step further by considering 
that on average, one performs $3/2$ evaluations as $e_1$ has a 50\% chance 
of being a direction of decrease. With that consideration, the expected 
decrease guarantee becomes 
\[
	\frac{2}{3\sqrt{\pi}}\frac{\Gamma(d/2)}{\Gamma(d/2+1/2)},
\]
which improves over the quantity~\eqref{eq:eDSF} for $p=1$.  As this result 
even holds for $p=d$, this provides a novel explanation for the performance of direct-search 
approaches using opportunistic polling (with or without random subspaces).  

\paragraph{Parallel processing:}
Using multiple cores to perform function evaluations in parallel is a common 
paradigm that affects the per-iteration workload. 
If $c$ cores are dedicated to distinct function evaluations and complete 
polling is performed, then one can consider that Algorithm~\ref{alg:DS} 
has an evaluation cost of $\lceil 2p/c \rceil$, where $\lceil \cdot \rceil$ 
denotes the ceiling function. Since the expected decrease $\Eds[p,d]$ is a 
decreasing function of $p$ (see, Corollary~\ref{cor:p12DS}) and assuming 
$2p/c$ is an integer number, Theorem~\ref{thm:expdecperfunDS} implies that
\begin{equation*}
	\frac{\Eds[p,d]}{\lceil 2p/c\rceil} 
	\ge c \frac{\Eds[p,d]}{2p} 
	> c \frac{\Eds[p+c/2,d]}{2(p+c/2)} 
	= \frac{\Eds[p+c/2,d]}{\lceil 2(p+c/2)/c\rceil}.
\end{equation*}
Consequently, in this parallel setting, the expected decrease per unit of work 
is maximized for $p=c/2$, i.e. the smallest subspace dimension that exploits 
all $c$ cores. Such a result shows that our analysis can be adapted to the 
computational power available to perform function evaluations.

\section{Analysis in the model-based setting}
\label{sec:theoMB}

In this section, we examine expected decrease for 
Algorithm~\ref{alg:MB}, i.e., when a model-based strategy is used to perform 
steps in the random subspace. The analysis is similar to that of 
Section~\ref{sec:theoDS} yet presents significant differences, as we will 
discuss below. Section~\ref{ssec:edMB} establishes the main expected decrease 
result, while Section~\ref{ssec:MBperfun} considers the results in light of  
per-iteration evaluation cost.

\subsection{Expected decrease formula}
\label{ssec:edMB}

We begin by deriving an expression for the expected decrease that does not 
depend on the selected basis for the random subspace.

\begin{proposition} 
\label{prop:EequalsMB}
	Consider the linear function $f^{lin}$ with $g \sim \Sd$, and suppose 
	that Algorithm~\ref{alg:DFAwRS} is applied using Algorithm~\ref{alg:MB} 
	as \algoclass{} (which we denote by \algoclass{}=\MB) with $\delta^k=1$. 
	Then, for any $k$, the expected decrease guarantee satisfies
	\begin{equation}
	\label{eq:EequalsMB}
		\Emb[p,d] = \EE[\tilde{g} \sim \Sd]{
		\sqrt{\sum_{i=1,\dots,p} \tilde{g}_i^2}
		}.
	\end{equation}		
\end{proposition}

\begin{proof}
As in the proof of Proposition~\ref{prop:EequalsDS}, we assume without loss of 
generality that $x^{k+1} \neq x^k$. Let $B=[b_1 \cdots b_p]$ with 
$b_i \in \R^d$. Since $\delta^k=1$, the simplex gradient calculated by 
Algorithm~\ref{alg:MB} is given by
\begin{equation*}
	\nabla_S f^{lin}|_p (x^k, {\tt I_p}) 
	=   I_{p\times p}
	\begin{bmatrix} 
		f^{lin}(x^k+ b_1) - f(x^k) \\ 
		\vdots \\ 
		f^{lin}(x^k+b_p) - f(x^k) 
	\end{bmatrix} 
	= B^{\top} g.
\end{equation*}
Therefore, the decrease obtained for $\delta^k=1$ is 
\begin{eqnarray*}
	f(x^k) - f(x^{k+1}) 
	&= &f(x^k) - f\left(x^k-B\frac{B^\T g}{\|B^\T g\|}\right) \nonumber \\
	&= &g^\T B \frac{B^\T g}{\|B^\T g} = \|B^\T g\|.
\end{eqnarray*}
In terms of expected decrease, we therefore obtain
\begin{equation*}
	\Emb[p,d] 
	= 
	\EE[\substack{B \sim \V_{p,d} \\ g \sim \V_{1,d}}]{\|B^\T g\|}.
\end{equation*}
By the same argument as in the proof of Proposition~\ref{prop:EequalsMB}, we 
can write $B=Q I_{d,p}$ with $Q \sim \V_{d,d}$ and $I_{d,p}$ containing the 
first $p$ coordinate directions in $\R^d$, and $Q^\T g$ is uniformly 
distributed in $\V_{1,d}$. This leads to
\begin{eqnarray*}
	\Emb[p,d] 
	&= &\EE[\substack{B \sim \V_{p,d} \\ g \sim \V_{1,d}}]{\|B^\T g\|} \\
	&= &\EE[\tilde{g} \sim \V_{1,d}]{\|I_{d,p}^\T \tilde{g}\|} \\
	&= &\EE[\tilde{g} \sim \V_{1,d}]{
	\sqrt{\sum_{i=1,\dots,p} \tilde{g}_i^2}},
\end{eqnarray*}
proving~\eqref{eq:EequalsMB}.	
\end{proof}
	
We now derive an expression for~\eqref{eq:EequalsMB}. Similarly to 
the direct-search case, when $d=p=1$, the expected decrease has a 
trivial expression
\[
    \Emb[1, 1] = 1.
\]

We assume in the rest of this section that $d>1$. In that case, the general 
form of the expected decrease is surprisingly elegant in that it does not 
include a trigonometric integral.

\begin{theorem}
\label{thm:expdecMB}
	Under the assumptions of Proposition~\ref{prop:EequalsMB}, suppose 
	further that $d>1$. Then, the expected decrease is given by
	\begin{equation}
	\label{eq:expdecMB}
    	\Emb[p,d] 
    	= 
    	\frac{\Gamma(d/2) ~\Gamma(p/2+1/2)}{\Gamma(d/2+1/2) ~\Gamma(p/2)}.
    \end{equation}
\end{theorem}

\begin{proof}
Our goal consists in evaluating the expression~\eqref{eq:EequalsMB}, i.e.
\[
    \Emb[p,d] = 
    \EE[\tilde{g} \sim \V_{1,n}]{\sqrt{\sum_{i=1,\dots,p} \tilde{g}_i^2}}. 
\]
Consider first the case $p=d$. Since $\tilde{g} \in \Sd$, we have
$\sqrt{\sum_{i=1}^p \tilde{g}^2_i}=\|\tilde{g}\|=1$, and thus
\[
	\Emb[d,d] = \EE[\tilde{g} \sim \V_{1,d}]{1} = 1.
\]
Noting that formula \eqref{eq:expdecMB} also returns $1$ when $p=d$ shows 
that it is valid in that case. Thus, in the rest of the proof, we 
suppose that $p<d$.

In order to compute the expectation, we restrict ourselves to vectors in the 
nonnegative orthant, i.e. we consider  
$R:=\{\tilde{g} \in \Sd\ |\ \tilde{g}_i \ge 0\ \forall i=1,\dots,d\}$. As in the 
proof of Theorem~\ref{thm:expdecDS}, we introduce hyperspherical coordinates
\begin{eqnarray*}
	x_d &= &\cos(\varphi_1), \\
	x_{d-1} &= &\sin(\varphi_1)\cos(\varphi_2), \\
	&\vdots & \\
	x_2 &= &\sin(\varphi_1) \cdots \sin(\varphi_{d-2})\cos(\varphi_{d-1}), \\
	x_1 &= &\sin(\varphi_1) \cdots \sin(\varphi_{d-2})\sin(\varphi_{d-1}),
\end{eqnarray*}
with surface element
\[
	dS = \sin^{d-2}(\varphi_1) \sin^{d-3}(\varphi_2) \cdots 
	\sin(\varphi_{d-2}) d\varphi_1 d\varphi_2 \cdots d\varphi_{d-1}.
\]
(As before, we use the reverse of the traditional ordering in order to create 
a simpler proof.) Then, for any $\tilde{g}$ in the nonnegative orthant, we 
have
\[
    \sqrt{\sum_{i=1}^{p} \tilde{g}^2_i} = \prod_{k=1}^{d-p} \sin(\varphi_k).
\]
Given that there are $2^d$ orthants in $R^d$, we obtain by symmetry that
\begin{equation*}
	\Emb[p,d] = \frac{2^d}{|\Sd|} 
    \int_{R} \left(\prod_{k=1}^{d-p} \sin(\varphi_k)\right) 
    \left(\prod_{k=1}^{d-2} \sin^{d-k-1}(\varphi_k)\right) 
    d\varphi_1 \cdots d\varphi_{d-1},
\end{equation*}
where $|\Sd|$ denotes the volume of the unit sphere in $\R^d$. By 
exploiting partial separability of this integral, we obtain
\begin{eqnarray}
    \Emb[p,d] &= &\frac{2^d}{|\Sd|} 
    \int_{R} \left(\prod_{k=1}^{d-p} \sin(\varphi_k)\right) 
    \left(\prod_{k=1}^{d-2} \sin^{d-k-1}(\varphi_k)\right) 
    d\varphi_1 \cdots d\varphi_{d-1}, \nonumber \\
    &= &\frac{2^d \ \pi }{2 |\Sd|} 
    \int_{R} \left(\prod_{k=1}^{d-p} \sin(\varphi_k)\right) 
    \left(\prod_{k=1}^{d-2} \sin^{d-k-1}(\varphi_k)\right) 
    d\varphi_1 \cdots d\varphi_{d-2}, \nonumber \\
    &= &\frac{2^d \ \pi}{2|\Sd|} 
    \left(\prod_{k=1}^{d-p} \int_{0}^{\pi/2}\sin^{d-k}(\theta) d\theta\right) 
    \left(\prod_{k=d-p+1}^{d-2} 
    \int_{0}^{\pi/2}\sin^{d-k-1}(\theta) d\theta\right). \nonumber\\
    \label{eq:edMBaux}
\end{eqnarray}

Recalling identity~\eqref{eq:sinusint}, we compute
\begin{eqnarray*}
    \prod_{k=d-p+1}^{d-2} \int_{0}^{\pi/2}\sin^{d-k-1}(\theta) d\theta 
    &= &\prod_{k=1}^{p-2} \int_{0}^{\pi/2}\sin^{p-k-1}(\theta) d\theta \\
    &= &\prod_{k=1}^{p-2} \int_{0}^{\pi/2} \sin^k(\theta) d\theta, \\
    &= &\prod_{k=1}^{p-2} \frac{\sqrt{\pi} 
    \ \Gamma(k/2+1/2)}{2 \ \Gamma(k/2+1)}, \\
    &= &\left(\frac{\sqrt{\pi}}{2}\right)^{p-2}\frac{1}{\Gamma(p/2)}.
\end{eqnarray*}
Also recalling~\eqref{eq:prodsinint} and substituting both 
into~\eqref{eq:edMBaux}, we find that
\begin{eqnarray*}
    \Emb[p,d] 
    &= &\frac{2^d \ \pi}{2 |\Sd|} 
    \left(\frac{\sqrt{\pi}}{2}\right)^{d-p} 
    \frac{\Gamma(p/2+1/2)}{ \Gamma(d/2+1/2)}
    \left(\frac{\sqrt{\pi}}{2}\right)^{p-2}\frac{1}{\Gamma(p/2)} \\
    &= &\frac{2 \ \pi^{d/2}}{|\Sd|} 
	\frac{\Gamma(p/2+1/2)}{ \Gamma(d/2+1/2)}
	\frac{1}{\Gamma(p/2)} \\
	&= &\frac{\Gamma(p/2+1/2)}{\Gamma(d/2+1/2)}
	\frac{\Gamma(d/2)}{\Gamma(p/2)},
\end{eqnarray*}
where the final line comes from the substitution 
$|\Sd|=2\pi^{d/2}/\Gamma(d/2)$. We have thus proved that~\eqref{eq:expdecMB} 
also holds in the case $p<d$, and the proof is complete.
\end{proof}

We examine several particular properties of the expression~\eqref{eq:expdecMB} 
in the next corollary. As in Section~\ref{ssec:edDS}, we leverage the fact 
that the expression~\eqref{eq:expdecMB} has a separable structure.

\begin{corollary}
\label{cor:p12MB}
	Let $d_1,d_2,p_1,p_2$ be integers greater than or equal to $1$ such that 
	$\max\{p_1,p_2\} \le \max\{d_1,d_2\}$. Then, the 
	following properties hold:
	\begin{enumerate}[(i)]
		\item $\Emb[1,d_1] 
		= \frac{1}{\sqrt{\pi}}\frac{\Gamma(d_1/2)}{\Gamma(d_1/2+1/2)}$;
		\item if $d_1>2$, then $\Emb[2,d_1] 
		= \frac{\sqrt{\pi}}{2} \frac{\Gamma(d_1/2)}{\Gamma(d_1/2+1/2)}$;
    	\item $\frac{\Emb[p_1,d_1]}{\Emb[p_2,d_1]} 
    	= \frac{\Emb[p_1,d_2]}{\Emb[p_2, d_2]}$;
    	\item $\frac{\Emb[p_1,d_1]}{\Emb[p_1,d_2]} 
    	= \frac{\Emb[p_2,d_1]}{\Emb[p_2,d_2]}$.
    \end{enumerate}
\end{corollary}

Notice that $\Emb[1,d]=\Eds[1,d]$ for any $d$, which should not come as 
a surprise since Algorithms~\ref{alg:DS} and~\ref{alg:MB} perform 
identically for $p=1$. Comparing $\Emb[2,d]$ and $\Eds[2,d]$, however, 
we observe that
\[
	\frac{\sqrt{2}}{\sqrt{\pi}} \approx 0.797 
	< 0.886 \approx \frac{\sqrt{\pi}}{2},
\]
implying that Algorithm \ref{alg:MB} is providing a higher expected decrease 
than Algorithm~\ref{alg:DS} when a two-dimensional subspace is used.

We end this subsection with asymptotic results akin to 
Corollary~\ref{cor:p12limitsDS}, that follows from combining 
Lemma~\ref{lem:asympt} with Corollary~\ref{cor:p12MB}.

\begin{corollary}
\label{cor:p12limitsMB}
	Under the same assumptions as Corollary~\ref{cor:p12MB}, 
	asymptotically 
	\begin{equation*}
		\Emb[1,d_1]
    	\rightarrow 
    	\frac{\sqrt{2}}{\sqrt{\pi}\sqrt{d_1}} 
    	~\mbox{as}~d_1\rightarrow \infty,
    \end{equation*}
    and
	\begin{equation*}
		\Emb[2,d_1] 
    	\rightarrow 
    	\frac{\sqrt{\pi}}{\sqrt{2}\sqrt{d_1}} 
    	~\mbox{as}~d_1\rightarrow \infty.
	\end{equation*}
\end{corollary}

\subsection{Expected decrease per function evaluation}
\label{ssec:MBperfun}

We now examine the expected decrease guarantee of Algorithm~\ref{alg:MB} by 
taking its function evaluation cost into account. While Algorithm~\ref{alg:DS} 
was evaluating $2p$ new points per iteration, Algorithm \ref{alg:MB} only
evaluates $p+1$ new points per iteration. Indeed, the construction of the 
simplex gradient requires $p+1$ function values but only $p$ new ones since 
that of the incumbent solution $x^k$ is re-used from the past iteration. One final evaluation is 
used in line 6 of Algorithm~\ref{alg:MB}, so the total amounts 
to $p+1$ new evaluations. As a result, we define 
\begin{equation}
\label{eq:eMBF}
	\Emb^F[p,d] = \frac{\Emb[p,d]}{p+1}
\end{equation}
for $p\ge 2$, and investigate its behavior as $p$ varies in 
Theorem~\eqref{eq:expdecperfunMB} (the case $p=1$ will be discussed separately). 

\begin{theorem}\label{thm:expdecperfunMB}
	Under the same assumptions as Theorem~\ref{thm:expdecMB}, suppose further 
	than $d> 2$. Then, for any $p=2,\dots,d-1$,
	\begin{equation}
	\label{eq:expdecperfunMB}
        \frac{\Emb[p,d]}{p+1} >  \frac{\Emb[p+1,d]}{p+2}.
	\end{equation}
\end{theorem}

\begin{proof}
To obtain the desired result, it suffices to prove that
\[
	\frac{\Emb[p,d]}{\Emb[p+1,d]} 
	= 
	\frac{\Gamma(p/2+1/2)^2}{\Gamma(p/2) \Gamma(p/2+1)} 
	> \frac{p+1}{p+2}.
\]
To this aim, we require a tighter version of Gautschi's 
inequality than the one used to prove Lemma~\ref{lem:intdecrDS}.  
By Kershaw's extension to Gautschi's inequality~\cite{DKershaw_1983}, for all $x>0$ and $s\in(0,1)$, it holds that
\begin{equation}
\label{eq:superGautschi}
    \left(x+s/2\right)^{1-s} 
    < \frac{\Gamma(x+1)}{\Gamma(x+s)} 
    < \left(x-1/2+(s+1/4)^{1/2}\right)^{1-s}.
\end{equation}

Applying $x=p/2$ and $s=1/2$ in equation~\eqref{eq:superGautschi} yields
\[
    \frac{\Gamma(p/2+1)}{\Gamma(p/2+1/2)} 
    < \frac{\sqrt{p+\sqrt{3}-1}}{\sqrt{2}}.
\]
Using $\Gamma(p/2+1)=(p/2)\Gamma(p/2)$, we also have
\[
    \frac{\Gamma(p/2)}{\Gamma(p/2+1/2)} 
    < \frac{\sqrt{2}\sqrt{p+\sqrt{3}-1}}{p}.
\]
Inverting both inequalities and multiplying the results shows that 
\[
    \frac{\Gamma(p/2+1/2)^2}{\Gamma(p/2)\Gamma(p/2+1)} 
    > \frac{p}{p+\sqrt{3}-1}.
\]
We can easily verify that $\frac{p}{p+\sqrt{3}-1} \geq \frac{p+1}{p+2}$ 
whenever $p\geq \sqrt{3}+1 \approx 2.73$. The case of $p=2$ is easily 
checked, as
\[
	\frac{\Emb[2,d]}{\Emb[3,d]}=\frac{\pi}{4}>\frac{2+1}{3+1},
\]
and therefore~\eqref{eq:expdecperfunMB} holds.
\end{proof}

The result of Theorem~\ref{thm:expdecMB} leads to similar conclusions than 
that of Theorem~\ref{thm:expdecperfunDS}, in the sense that using 
low-dimensional subspace dimension leads to better expected decrease 
guarantees up to $p\ge 2$. We comment thereafter on other settings.

\paragraph{The case $p=1$:}
The inequality~\eqref{eq:expdecperfunMB} does not apply for $p=1$, as 
\[
	\frac{\Emb[1,d]}{\Emb[2,d]} = \frac{2}{\pi} < \frac{1+1}{2+1},
\]
seemingly indicating that $p=2$ is the best choice. However, when $p=1$, 
the simplex gradient is necessarily equal to $b_1$ or $-b_1$. In the former
case, Algorithm~\ref{alg:MB} will not require an additional value on line 6, 
since the value at $z+\delta u=z+\delta b_1$ was already computed and used 
to form the simplex gradient. As a result, the average number of function 
evaluations used when $p=1$ is $3/2$ (similar to the case of opportunistic 
polling discussed in Section~\ref{ssec:DSperfun}). By extending~\eqref{eq:eMBF} 
to $p=1$ using this cost, we obtain
\begin{equation*}
	\Emb^F[1,d] := \frac{\Emb[1,d]}{3/2} 
	= \frac{2}{3\sqrt{\pi}}\frac{\Gamma(d/2)}{\Gamma(d/2+1/2)} 
	> \Emb^F[2,d],
\end{equation*}
suggesting that one-dimensional subspaces also provide 
a better return on investment in model-based approaches based on simplex 
gradients, i.e., linear models of the function.

\paragraph{Parallel processing:}
Similarly to the direct-search case, we can consider the situation where $c$ 
parallel cores are used to compute distinct function evaluations. This 
paradigm reduces the per-iteration cost of Algorithm~\ref{alg:MB} to 
$\lceil p/c \rceil +1$, where the gain is necessarily achieved only on the 
evaluations used to form the simplex gradient (the final evaluation on 
Line 6 must be done after the others). Then, assuming $p/c$ is an integer, 
we obtain
\begin{equation*}
	\frac{\Emb[p,d]}{\lceil {p}/c\rceil + 1} 
	= c\frac{\Emb[{p},d]}{p + c}
    \quad\mbox{and}\quad
    \frac{\Emb[p+c,d]}{\lceil (p+c)/c\rceil + 1} 
    = c\frac{\Emb[p+c,d]}{p + 1 + c}.
\end{equation*}
Although the result of Theorem~\ref{thm:expdecperfunMB} does not directly 
apply to this new quantity (unless $c=1$), a simple numerical inspection
confirms that $\frac{\Emb[p,d]}{\lceil {p}/c\rceil + 1}$ is maximized for 
$p=c$ for all values $c \in \{1, 2, \ldots, 256\}$ and 
$p \in \{c, 2c, \ldots, 100c\}$.  This strongly suggests that the expected 
decrease per unit of work is maximized when you use the smallest subspace 
that uses all cores, as in the direct-search setting. However, this maximum 
is not uniquely obtained, since when $d\ge 4$ and $c=2$, we have
\begin{equation*}
    \frac{\Emb[2,d]}{\lceil 2/2\rceil + 1} 
    = \frac{\sqrt{\pi}}{4} \cdot \frac{\Gamma(d/2)}{\Gamma(d/2+1/2)} 
    = \frac{\Emb[4,d]}{\lceil 4/2\rceil + 1},
\end{equation*}
hence both $p=2$ and $p=4$ achieve the maximum expected decrease per unit of 
work.


\section{Numerical estimation of expected decrease}
\label{sec:num}

In Sections \ref{sec:theoDS} and \ref{sec:theoMB}, we showed that the expected 
decrease per function evaluation is strictly decreasing as a function of $p$.  Considering 
Corollaries~\ref{cor:p12DS} and \ref{cor:p12MB}, we see that the 
expected decrease improves from $p=1$ to $p=2$.  Indeed, 
\[
	\frac{\Emb[2,d]}{\Emb[1,d]} = \frac{\pi}{2}  > 1 
        \quad \mbox{and} \quad 
	\frac{\Eds[2,d]}{\Eds[1,d]} = \sqrt{2} > 1.
\]
However, the expected decrease per function evaluation actually worsens from 
$p=1$ to $p=2$.  Indeed,
\[
	\frac{\Eds^F[2,d]}{\Eds^F[1,d]} = \sqrt{2}/2 < 1 
        \quad \mbox{and} \quad 
	\frac{\Emb^F[2,d]}{\Emb^F[1,d]} = \frac{\pi}{4} < 1.
\]
Computing these ratios becomes increasingly cumbersome as $p$ increases. In 
this section, we thus investigate the behavior of the expected decrease 
quantities $\Eds$, $\Emb$, $\Eds^F$ and $\Emb^F$ numerically, by way of to 
Monte Carlo simulations.

Algorithm~\ref{alg:MCestim} describes our estimation procedure applied to 
evaluate the expected decrease quantities. Note that it samples both a 
vector $g$ uniformly distributed on the unit sphere and a random basis 
$B$ for the subspace, as in the original definition~\eqref{eq:Epn}.  
(As such, we also numerically verify the results in Proposition \ref{prop:EequalsDS} and \ref{prop:EequalsMB}.) 
The estimated quantity is obtained by averaging the decrease formulas for 
every sample $(g,B)$. In our subsequent experiments, we use 
$N_{\tt sims} = 10^{4}$ samples.

\begin{algorithm}[ht!]
\caption{Monte Carlo estimation of expected decrease {\tt MCestim})}
\label{alg:MCestim}
\begin{algorithmic}[1]
\Procedure{MCtest}{$N_{\tt sims}$, $b$, $p$, \algoclass{}}
    \State \% $N_{\tt sims}$ number of simulations to run, positive integer
    \State \% $d$ problem dimension, positive integer
    \State \% $p$ subspace dimension, $p \in \{1, 2, \ldots, d\}$
    \State \% \algoclass{}: DFO step on subspaces
    \algoclass{} $\in$ \{\DS{},\MB{}\}
    \For{$k = 1$ to $N_{\tt sims}$}
    \State Randomly select $g \in \Sd$
    \State Randomly select a subspace of dimension $p$ with orthonormal basis 
    $$B = [b_1, b_2, ..., b_p]$$
    \If{\algoclass{}=\DS{}}
    	\State Set $D(k) = \max\{ g^\top d : d = \pm b_i, i = 1, 2, \ldots, p\}$
    \Else
    	\State Compute the subspace gradient $\hat{g} = B^\top g$
   		\State Set $D(k) = (-B(\hat{g}/\|\hat{g}\|))^\top g$
    \EndIf 
    \EndFor
    \State Return $\sum_{k=1}^{N_{\tt sims}}D(k)/N_{\tt sims}$ as an estimate 
    of $\mathbb{E}_{\algoclass}[p,d]$
\EndProcedure
\end{algorithmic}
\end{algorithm}

\subsection{Direct-search case}
\label{ssec:numDS}

We first look at the results for estimating $\Eds$. Note that we can compute 
the integral symbolically for low values of $p$ using 
Mathematica~\cite{Mathematica}, yielding
\begin{small}
\begin{equation*}
	\begin{array}{lllll}
    	\Eds[3,d] 
    	&= &\frac{\Gamma(d/2)}{\Gamma(d/2+1/2)} 
    	\left[\frac{12\arctan(\sqrt{2})+3\arctan(460\sqrt{2}/329)}
    	{2\sqrt{2}(\sqrt{\pi})^3}\right]
    	&\approx &0.938 \frac{\Gamma(d/2)}{\Gamma(d/2+1/2)}, \\
    	\Eds[4,d] 
    	&= &\frac{\Gamma(d/2)}{\Gamma(d/2+1/2)} 
    	\left[\frac{12\sqrt{2}\arctan(\frac{1}{2\sqrt{2}})}
    	{(\sqrt{\pi})^3}\right]
    	&\approx &1.036 \frac{\Gamma(d/2)}{\Gamma(d/2+1/2)}.
    \end{array}
\end{equation*}
\end{small}
Further estimation of the Gamma functions leads to the approximations
\begin{equation*}
	\frac{\Eds^F[3,d]}{\Eds^F[2,d]} 
	\approx 0.784 
	\quad \mbox{and} \quad
	\frac{\Eds^F[4,d]}{\Eds^F[3,d]}
	\approx 0.828.	
\end{equation*}
These values suggest that the gain in expected decrease between $p$ and 
$p+1$ reduces as the value of $p$ increases.


Numerical estimations of the expected decrease for direct-search are given in 
Figure~\ref{fig:DS}.
\footnote{In all figures it should be recognized that lines adjoining points are for visualization only.  The values of $d$ and $p$ are always integers.}
Figure~\ref{fig:DSvarydimension} presents the output of 
Algorithm~\ref{alg:MCestim} (with \algoclass{}=\DS{}) for varying dimensions 
$d\in\{8,16,32,\ldots,1024\}$ using subspace dimension $p\in \{1, 2, d/2, d\}$. 
For $p\in\{1,2\}$ we superimpose the exact formula for the expected decrease 
as given by Corollary \ref{cor:p12DS}.  For large values of $d$, 
floating-point and overflow errors occur when evaluating 
$\Gamma(d/2)/\Gamma(d/2+1/2)$, thus we only plot the values from 
Corollary~\ref{cor:p12DS} up to occurrence of these errors. For comparison, 
we also show the large-$d$ asymptotic results from 
Corollary~\ref{cor:p12limitsDS}. We note that the Monte-Carlo simulation 
aligns nearly perfectly with the formulas for $p \in \{1,2\}$, while the 
large-$d$ asymptotics are essentially indistinguishable from the simulations 
for $d \ge 100$. 

Figure~\ref{fig:DSvarysubspacesize} shows the output of 
Algorithm~\ref{alg:MCestim} (with \algoclass{}=\DS{}) for varying subspace 
size $p\in\{1,2,3,4,5,10,20,50,100,200,500,1000\}$ and fixed dimension 
$d=1000$. As expected, we observe that choosing $p=1$ provides the 
worst expected decrease and that $p=d$ leads to the best expected decrease. 
Note also that the expected decrease diminishes as $d$ increases.

\begin{figure}[ht]
  \centering
  \begin{subfigure}[b]{0.45\textwidth}
    \includegraphics[width=\textwidth]{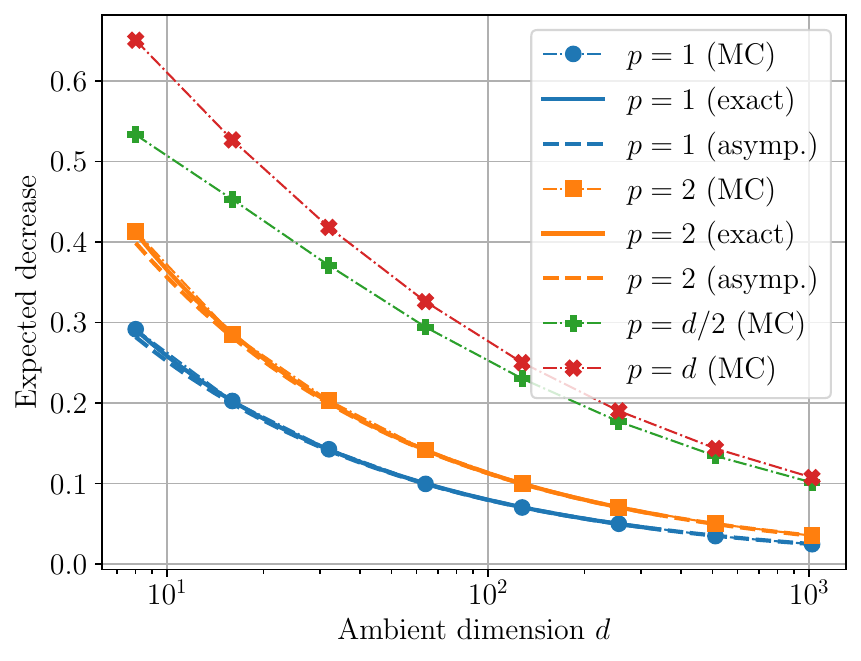}
    \caption{Varying dimension $d$}\label{fig:DSvarydimension}
  \end{subfigure}
  ~
  \begin{subfigure}[b]{0.45\textwidth}
    \includegraphics[width=\textwidth]{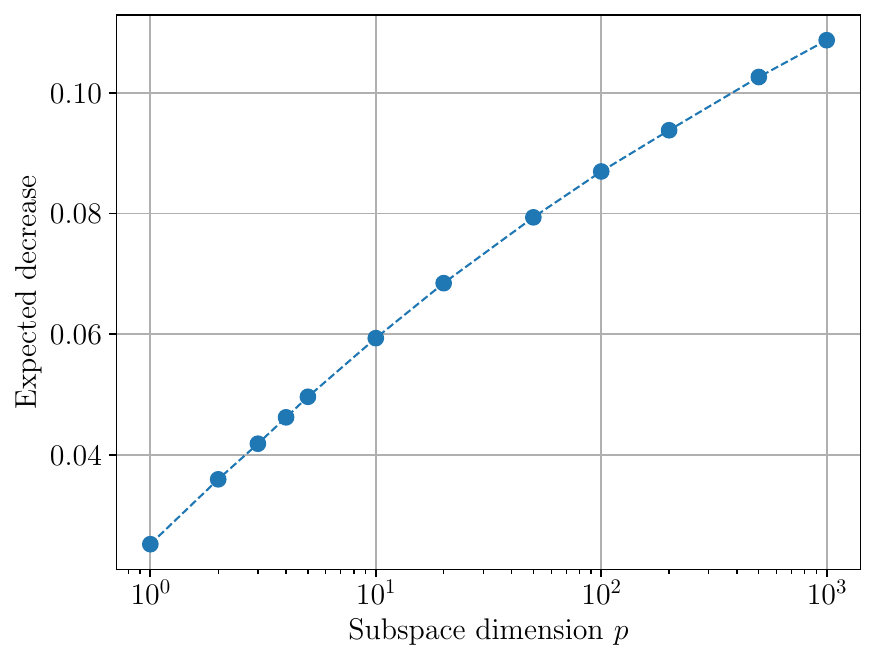}
    \caption{Varying subspace dimension $p$}\label{fig:DSvarysubspacesize}
  \end{subfigure}
  \caption{Expected decrease ($\Eds[p,d]$) versus average decrease based on 
  Monte Carlo simulation for varying dimension (a) and subspace dimension (b).
  \newline
  {\footnotesize (a) Lines with ``(MC)'' are the Monte Carlo simulation results, 
  ``(exact)'' is the result from Theorem \ref{thm:expdecDS} and ``(asymp.)'' is 
  the large-$d$ asymptotic result from Corollary~\ref{cor:p12limitsDS}.}
  \newline
  {\footnotesize(b) Ambient dimension $d=1000$.}
  }
  \label{fig:DS}
\end{figure}

In Theorem \ref{thm:expdecperfunDS}, we showed that the expected decrease per 
function evaluation $\Eds^F[p,d]$ was strictly decreasing as a function of the 
subspace dimension $p$. In Figure \ref{fig:DSratio}, we plot the expected 
decrease per unit work for varying dimensions and varying subspace dimensions. 
Those results confirm our theoretical findings, in that setting $p=1$ gives 
the largest expected decrease per function evaluation. Note that the gap between 
$\Eds^F[p,d]$ and $\Eds^F[p,d]$ is the largest for $p=1$, and that it decreases 
as $p$ increases.

\begin{figure}[ht]
  \centering
  \begin{subfigure}[b]{0.45\textwidth}
    \includegraphics[width=\textwidth]{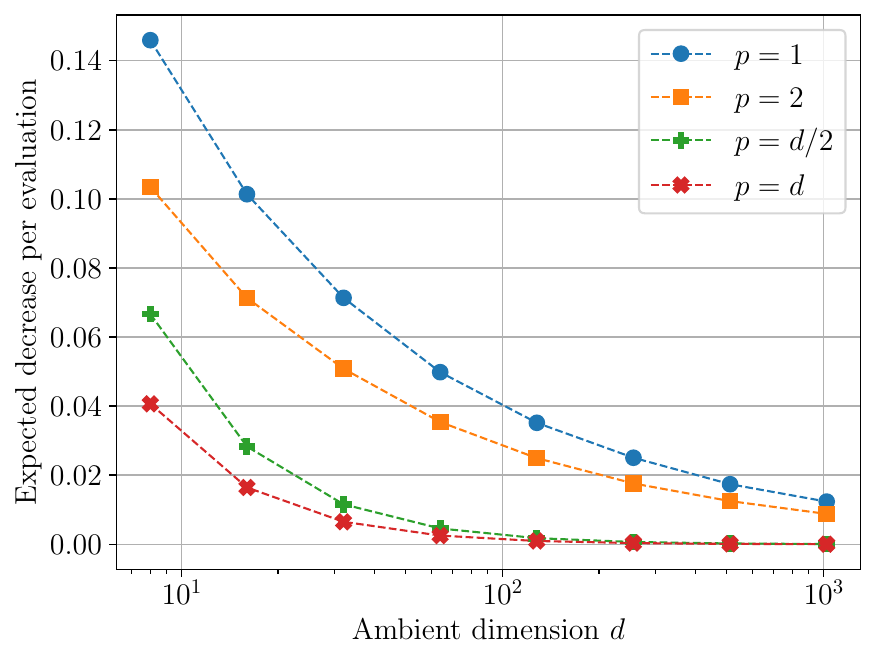}
    \caption{Varying dimension $d$}\label{fig:DSratiovarydimension}
  \end{subfigure}
  ~
  \begin{subfigure}[b]{0.45\textwidth}
    \includegraphics[width=\textwidth]{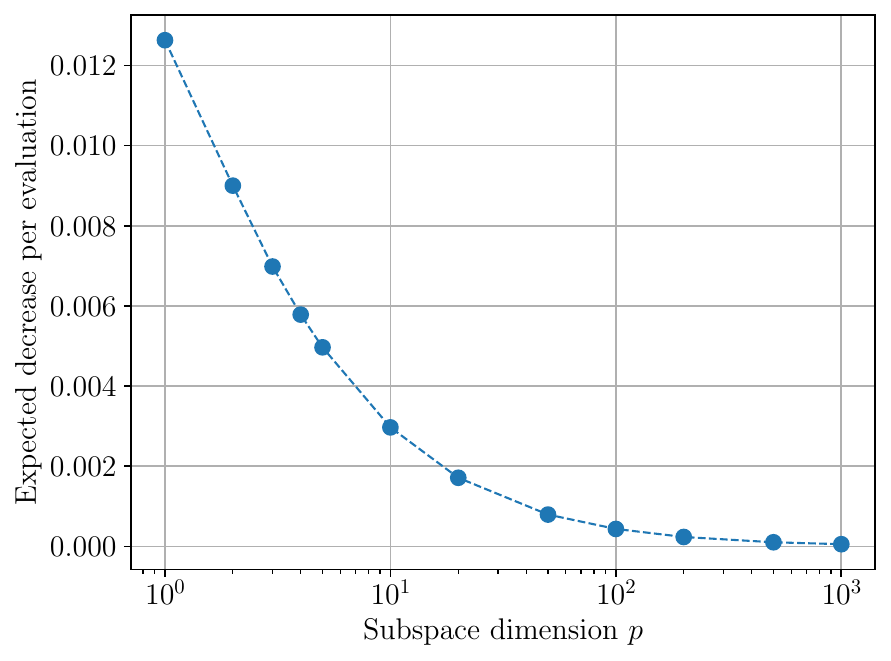}
    \caption{Varying subspace dimension $p$}\label{fig:DSratiovarysubspacesize}
  \end{subfigure}
  \caption{Expected decrease per function evaluation 
  $\Eds^F[p,d]$~\eqref{eq:expdecDS} versus average decreased based on Monte 
  Carlo simulation for varying dimension (a) and subspace dimension (b).
  }
  \label{fig:DSratio}
\end{figure}

\subsection{Model-based case}
\label{ssec:mb}

We now discuss the output of Algorithm~\ref{alg:MCestim} using 
\algoclass{}=\MB{}. Figures~\ref{fig:MB} and~\ref{fig:MBratio} present results 
analogous to that of Figures~\ref{fig:DS} and \ref{fig:DSratio}.    

\begin{figure}[ht]
  \centering
  \begin{subfigure}[b]{0.45\textwidth}
    \includegraphics[width=\textwidth]{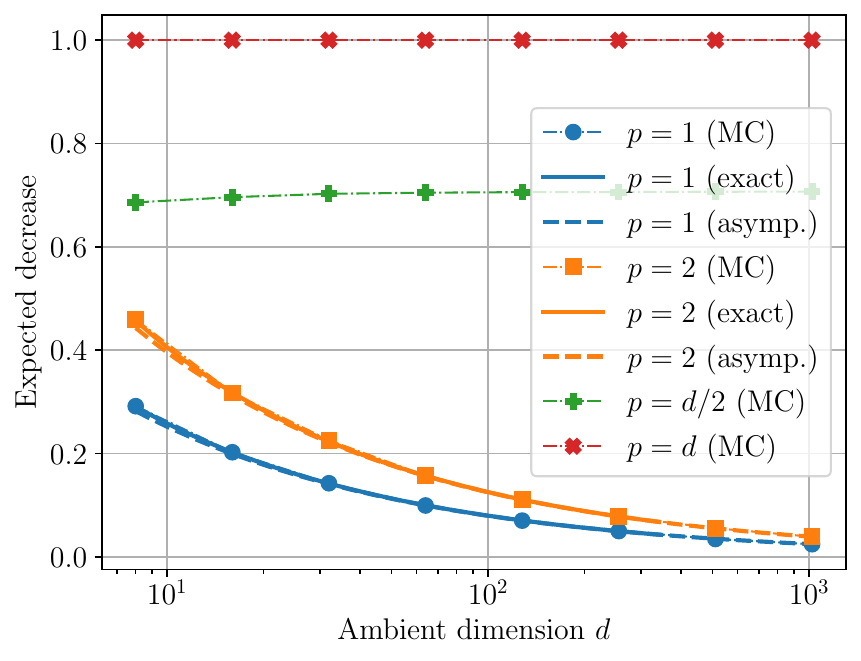}
    \caption{Varying dimension $d$}
  \end{subfigure}
  ~
  \begin{subfigure}[b]{0.45\textwidth}
    \includegraphics[width=\textwidth]{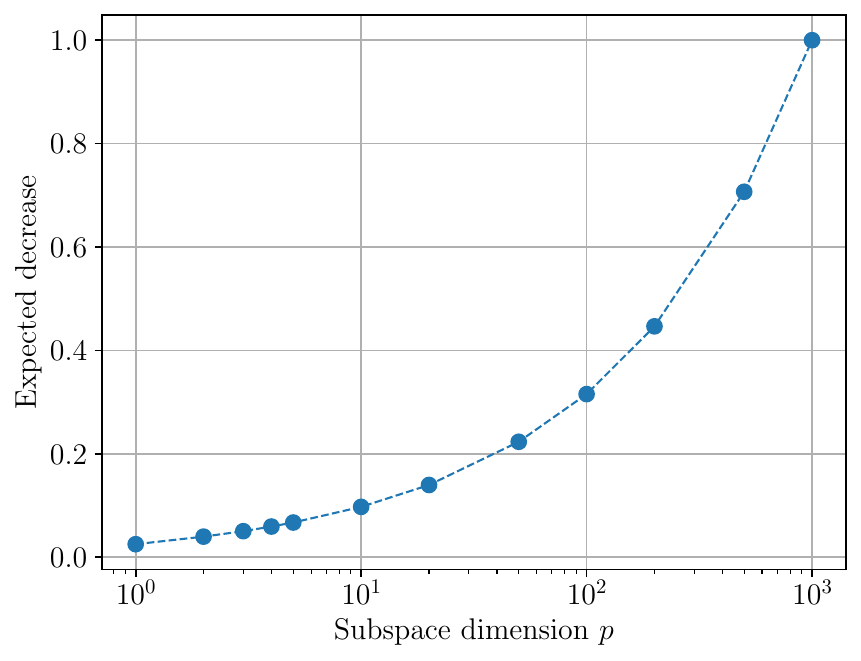}
    \caption{Varying subspace dimension $p$}
  \end{subfigure}
  \caption{Expected decrease ($\Emb[p,d]$) versus the average decreased based on Monte 
  Carlo simulation for varying dimension (a) and subspace dimension (b).}
  \label{fig:MB}
\end{figure}

\begin{figure}[ht]
  \centering
  \begin{subfigure}[b]{0.45\textwidth}
    \includegraphics[width=\textwidth]{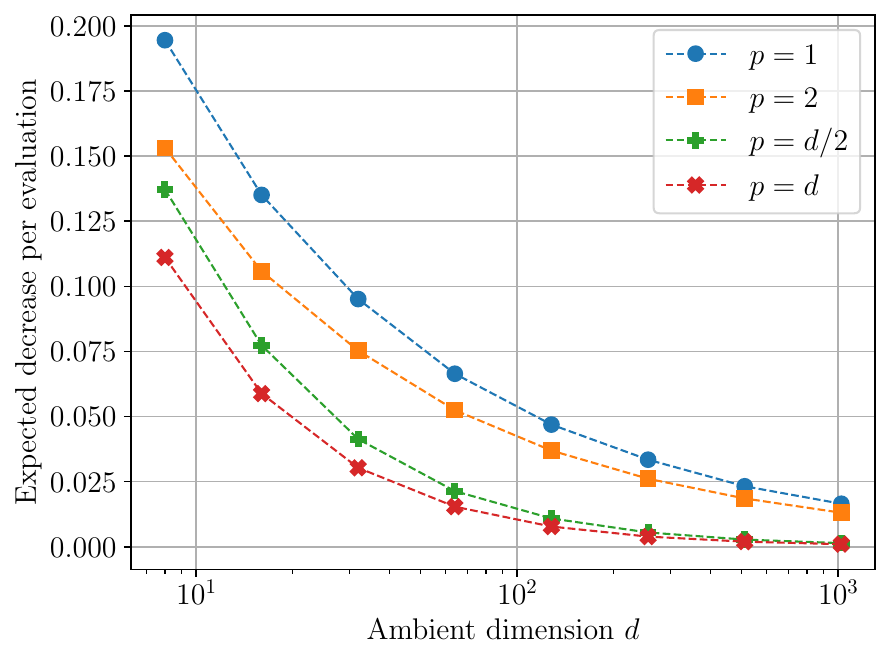}
    \caption{Varying dimension $d$}
    \label{fig:MBratiovarydimension}
  \end{subfigure}
  ~
  \begin{subfigure}[b]{0.45\textwidth}
    \includegraphics[width=\textwidth]{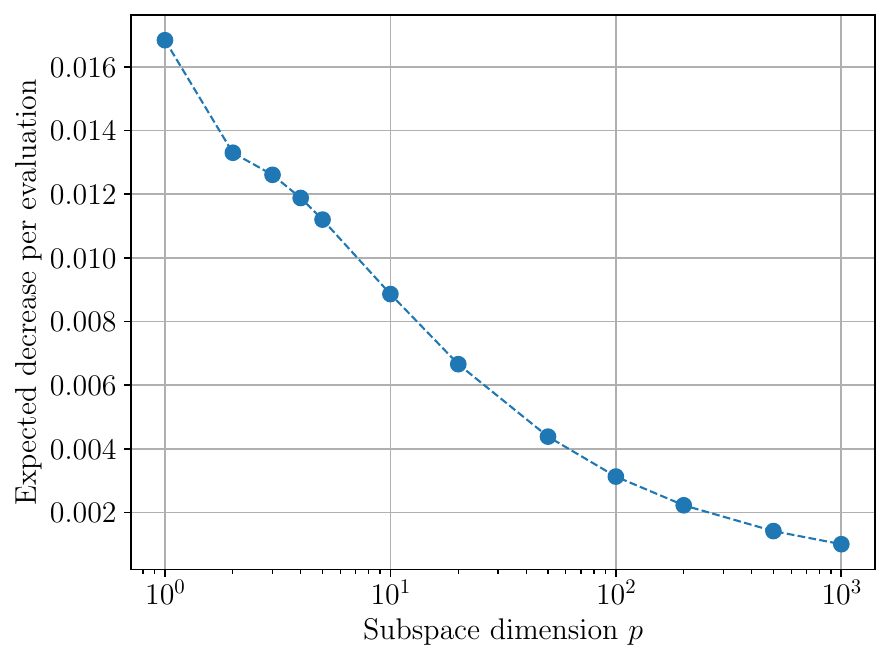}
    \caption{Varying subspace dimension $p$}
    \label{fig:MBratiovarysubspacesize}
  \end{subfigure}
  \caption{Expected decrease per unit work $\Emb^F[p,d]$ versus the average 
  based on Monte Carlo simulation for varying dimension (a) and subspace 
  dimension (b).}
  \label{fig:MBratio}
\end{figure}

As in the direct-search case, we match exact results for $p \in \{1,2\}$ (see 
Corollary~\ref{cor:p12MB}) and large $d$-asymptotics (see 
Corollary~\ref{cor:p12limitsMB}) quite closely. We also observe empirically 
that $p=1$ is worst in terms of expected decrease but best in terms of 
expected decrease per function evaluation (with our choice of 
$\Emb^F[1,d]=\Emb[1,d]/(3/2)$ explained in Section~\ref{ssec:MBperfun}). 
Finally, we see from Figure~\ref{fig:MBratiovarysubspacesize} that the gap 
between $p=1$ and $p=2$ is the largest among all consecutive values of $p$.

\section{Discussion}
\label{sec:conc}

We have established expected decrease formulae for derivative-free 
iterations using random subspaces when applied to linear functions. As 
explained in Section~\ref{ssec:expdec}, our analysis can be employed to show 
expected decrease guarantees for more general classes of smooth functions that 
admit a linear model approximation.
We have established that performing iterations of derivative-free algorithms 
in randomly generated subspaces is more beneficial as the dimension of the 
subspaces decreases. This arguably surprising result arises from properties of 
the uniform distribution over subspaces, and goes some way to understanding the strong empirical performance of low-dimensional subspace approximations (e.g.~in \cite{SGratton_CWRoyer_LNVicente_ZZhang_2015,LRoberts_CWRoyer_2023}).

Extending our analysis to handle quadratic models is a natural continuation of 
this paper, that poses a number of challenges related to the theory of random 
quadratic functions. Nevertheless, such results seem necessary to understand 
derivative-free methods that rely on quadratic models and beyond. In addition, 
elaborate implementations of derivative-free algorithms can reuse past 
evaluations to produce better trial points, which introduces non-trivial 
dependencies between iterations. Finally, we expect our theory to apply in 
the case of stochastic function evaluations, provided those satisfy common 
probabilistic properties appearing in the literature.

\bibliographystyle{plain} 
\bibliography{refsLDS}

\end{document}